\documentclass[final,a4paper]{siamltex1213}
  \usepackage{amsmath}
  \usepackage{amssymb}
 \usepackage{a4wide}
  \usepackage{paralist}
  \usepackage{subfigure}
    \usepackage{color}

\newtheorem{remark}{Remark}

\newcommand {\be}{\begin{equation}}
\newcommand {\ee}{\end{equation}}  
\newcommand {\eps} {\varepsilon}
\newcommand {\alp} {\alpha}
\newcommand {\la} {\lambda}

\newcommand {\deri}[2] {\frac {\partial #1}{\partial #2}}

\newcommand {\pp}{\boldsymbol{p}}
\newcommand {\qq}{\boldsymbol{q}}
\newcommand {\rr}{\boldsymbol{r}}
\newcommand {\ff}{\boldsymbol{f}}
\newcommand {\UU}{\boldsymbol{U}}

\title{The ``exterior approach" applied to the inverse obstacle problem for the heat equation}

\author{Laurent Bourgeois\thanks{Laboratoire POEMS, ENSTA ParisTech, 
 828, Boulevard des Mar\'echaux, 91762 Palaiseau Cedex, France} ({\tt laurent.bourgeois@ensta.fr})
        \and J\'er\'emi Dard\'e 
 \thanks{Institut de Math\'ematiques, Universit\'e de Toulouse, F-31062 Toulouse Cedex 9, France}}

\begin{document}
\maketitle

\begin{abstract} 
In this paper we consider the ``exterior approach" to solve the inverse obstacle problem for the heat equation.
This iterated approach is based on a quasi-reversibility method to compute the solution from the Cauchy data while a simple level set method is used to  
characterize the obstacle. We present several mixed formulations of quasi-reversibility that enable us to use some classical conforming finite elements. Among these, an iterated formulation
that takes the noisy Cauchy data into account in a weak way is selected to serve in some numerical experiments and show the feasibility of our strategy of identification.
\end{abstract}


\pagestyle{myheadings}
\thispagestyle{plain}
\markboth{The inverse obstacle problem for the heat equation}{}

\section{Introduction}
This paper deals with the inverse obstacle problem for the heat equation, which can be described as follows. We consider a bounded domain $D \subset \mathbb{R}^d$, $d \geq 2$, which contains
an inclusion $O$. The temperature in the complementary domain $\Omega=D \setminus \overline{O}$ satisfies the heat equation while the  
inclusion is characterized by a zero temperature.
The inverse problem consists, from the knowledge of the lateral Cauchy data (that is both the temperature and the heat flux) on a subpart of the boundary $\partial D$ during a certain interval of time $(0,T)$ such that the temperature at time $t=0$ is $0$ in $\Omega$, to identify the inclusion $O$.
Such kind of inverse problem arises in thermal imaging, as briefly described in the introduction of \cite{bryan_caudill}.
The first attempts to solve such kind of problem numerically go back to the late 80's, as illustrated by \cite{banks_kojima_winfree}, in which a least square method based on a shape derivative technique is used and numerical applications in 2D are presented. A shape derivative technique is also used in \cite{chapko_kress_yoon} in a 2D case as well, but the least square method is replaced by a 
Newton type method. Lastly, the shape derivative together with the least square method have recently been used in 3D cases \cite{harbrecht_tausch}.
The main feature of all these contributions is that they rely on the computation of forward problems in the domain $\Omega \times (0,T)$: this computation 
obliges the authors to know one of the two lateral Cauchy data (either the temperature or the heat flux) on the whole boundary $\partial D$ of $D$.
In \cite{ikehata_kawashita}, the authors introduce the so-called ``enclosure method", which enables them to recover an approximation of the convex hull of the inclusion without computing any forward problem. Note however that the lateral Cauchy data has to be known on the whole boundary $\partial D$. 

The present paper concerns the ``exterior approach", which is an alternative method to solve the inverse obstacle problem.
Like in \cite{ikehata_kawashita}, it does not need to compute the solution of the forward problem and in addition, it is applicable even if the lateral Cauchy data are known only on a subpart of $\partial D$, while no data are given on the complementary part. The ``exterior approach" consists in defining a sequence of domains that converges in a certain sense to the inclusion we are looking for. More precisely, the $n$th step consists,
\begin{enumerate}
\item for a given inclusion $O_n$, in approximating the temperature in $\Omega_n \times (0,T)$ ($\Omega_n:=D \setminus \overline{O_n}$) with the help of a quasi-reversibility method,  
\item for a given temperature in $\Omega_n \times (0,T)$, in computing an updated inclusion $O_{n+1}$ with the help of a level set method.
\end{enumerate}
Such ``exterior approach" has already been successfully used to solve inverse obstacle problems for the Laplace equation \cite{bourgeois_darde1,bourgeois_darde2,darde_finlandais} and for the Stokes system \cite{bourgeois_darde3}. It has also been used for the heat equation in the 1D case \cite{becache_bourgeois_darde_franceschini}: the problem in this simple case might be considered as a toy problem since the inclusion reduces to a point in some bounded interval. The objective of the present paper is to extend the ``exterior approach" for the heat equation to any dimension of space, with numerical applications in the 2D case.
Let us shed some light on the two steps of the ``exterior approach". In the step $1$, the quasi-reversibility method is used to 
approximate the solution to the heat equation with lateral Cauchy data and zero initial condition in a fixed domain, which is a linear ill-posed problem.
Quasi-reversibility was first introduced in \cite{lattes_lions} and roughly speaking consists, in its classical form, of a Tikhonov regularization applied to a bounded heat operator $\partial_t. -\Delta.$ from a Hilbert space to another, so that such operator is injective with dense range, but not surjective. Quasi-reversibility has been then extensively studied by M. V. Klibanov 
\cite{klibanov}.
The main feature of this non iterative method is that it can be directly interpreted as a weak formulation and hence is well adapted to a finite element method.
However, such weak formulation corresponds to a fourth-order problem and requires some Hermite type finite elements \cite{bourgeois_darde1}.
This was our main motivation to introduce some mixed formulations of quasi-reversibility in order to replace the fourth-order problem by a system of two second-order problems 
for which Lagrange finite elements are sufficient. Those mixed formulations were introduced
first for the Laplace equation \cite{bourgeois,darde_finlandais},
then for the Stokes system \cite{bourgeois_darde3} and eventually for the heat/wave equations \cite{becache_bourgeois_darde_franceschini}.
Besides, our mixed formulations have some communalities with the regularization methods recently proposed in \cite{burman,cindea_munch}.
In the step $2$ of the ``exterior approach", we use a non standard level set method based on the resolution of a Poisson equation instead of a traditional eikonal equation.
Its main advantage is that the Poisson equation can be solved with the help of simple Lagrange finite elements based on the mesh that is already used for the quasi-reversibility method.
Such level set method was first used 
for the Laplace equation in \cite{bourgeois_darde1,bourgeois_darde2,darde_finlandais}
and for the Stokes system in \cite{bourgeois_darde3}.

The article is organized as follows. In section 2 we introduce our inverse obstacle problem as well as several uniqueness results.
Section 3 is dedicated to different mixed formulations of quasi-reversibility to solve the heat equation with lateral Cauchy data and initial condition. 
In section 4 we present our algorithm to solve the inverse obstacle problem.
Numerical experiments are eventually shown in section 5.
\section{The statement of the inverse problem and some uniqueness results} 
\subsection{Statement of the problem}
\label{statement}
Let $D$ and $O$ be two open bounded domains of $\mathbb{R}^d$, with $d \geq 2$, 
the boundaries $\partial D$ and $\partial O$ of which are Lipschitz continuous. 
The domains $D$ and $O$ are such that $\overline{O} \subset D$ and $\Omega:=D \setminus \overline{O}$ is connected.
Let $\Gamma$ be an open non empty subset of $\partial D$. 
For some real $T>0$ and a pair of lateral Cauchy data $(g_0,g_1)$ on $\Gamma \times (0,T)$, 
the inverse obstacle problem consists to find $O$ such that for some function 
$u \in L^2(0,T;H^1(\Omega))$:
\begin{equation}
\label{obstacle}
\left\{
\begin{array}{cccc}
& \partial_t u -\Delta u  = 0&  \text{in} &\Omega \times (0,T)\\
& u=g_0  & \text{on} & \Gamma \times (0,T)\\
& \partial_\nu u=g_1 & \text{on} & \Gamma \times (0,T)\\
& u=0 & \text{on} &\partial O \times (0,T)\\
& u=0 & \text{on} &\Omega \times \{0\},
\end{array}
\right.
\end{equation}
where $\nu$ is the unit outward normal vector on $\partial \Omega$. 
We have to cope with problem (\ref{obstacle}) if for example, starting from a zero temperature in the domain $\Omega$, we try to recover the inclusion $O$ by imposing a heat flux $g_1$ on the accessible 
part of the boundary $\Gamma$ during the interval of time $(0,T)$ and by measuring the resulting temperature on the same part of the boundary during the same interval of time.
On the inaccessible part of the boundary $\partial D \setminus \overline{\Gamma}$, no data is provided.
The assumption $u \in L^2(0,T;H^1(\Omega))$ is sufficient to properly define all the boundary conditions of problem (\ref{obstacle}).
First of all, $u|_{\Gamma \times (0,T)} \in L^2(0,T;H^{1/2}(\Gamma))$ and $u|_{\partial O \times (0,T)} \in L^2(0,T;H^{1/2}(\partial O))$.
Let us now introduce the set $Q=\Omega \times (0,T)$ and the vector field $\UU \in \mathbb{R}^{d+1}$ defined in $Q$ by
\begin{equation} \UU=(\nabla u,-u)=(\partial_{x_i}u,-u),\quad i=1,\cdots,d.\label{vector}\end{equation}
We clearly have
${\rm div}_{d+1}\UU=\Delta u - \partial_t u=0$ in $Q$, which implies that $\UU \in H_{{\rm div},Q}:=\{\UU \in (L^2(Q))^{d+1},\,{\rm div}_{d+1}\UU \in L^2(Q)\}$. As a consequence
we have $\UU\cdot \nu_{d+1} \in H^{-1/2}(\partial Q)$, where $\nu_{d+1}$ is the unit outward normal on $\partial Q$. In particular, we conclude that
$\partial_\nu u|_{\Gamma \times (0,T)} \in H^{-1/2}(\Gamma \times (0,T))$ and $u|_{\Omega \times \{0\}} \in H^{-1/2}(\Omega \times \{0\})$.
\begin{remark}
It is important to note that since in (\ref{obstacle}) we have no boundary condition on the complementary part of $\Gamma$ in $\partial D$, we cannot benefit from some ``hidden regularity" and more generally from classical regularity results for the heat equation.
\end{remark}
\subsection{A uniqueness result}
Uniqueness for problem (\ref{obstacle}) is well-known (see for example \cite{chapko_kress_yoon}).
However we state and prove the result for sake of self-containment and in order to be precise on regularity assumptions.
\begin{theorem}
\label{unique}
For $i=1,2$, let two domains $O_i$ and corresponding functions $u_i \in L^2(0,T;H^1(\Omega_i))$ satisfy problem (\ref{obstacle}) with data $(g_0,g_1) \neq 0$.
Assume in addition that $u_i \in L^2(0,T;C^0(\overline{\Omega_i}))$.
Then we have $O_1=O_2$ and $u_1=u_2$.
\end{theorem}
\begin{proof}
Let us consider $\tilde{\Omega}$ the connected component of $D \setminus \overline{O_1 \cup O_2}$ which is in contact with $\Gamma$, and $\tilde{O}:=D \setminus \overline{\tilde{\Omega}}$.
The function $u=u_1-u_2$ satisfies in $\tilde{\Omega}$ the problem
\[
\left\{
\begin{array}{cccc}
& \partial_t u -\Delta u  = 0&  \text{in} & \tilde{\Omega} \times (0,T)\\
& u=0  & \text{on} & \Gamma \times (0,T)\\
& \partial_\nu u=0 & \text{on} & \Gamma \times (0,T).
\end{array}
\right.
\]
By Holmgren's theorem we obtain that $u$ vanishes in $\tilde{\Omega}$, that is $u_1=u_2$ in $\tilde{\Omega}$. Assume that 
$O_2$ is not contained in  $O_1$, which implies that
the open domain $R_2=\tilde{O} \setminus \overline{O_2}$ is not empty.
We have $u_2=0$ on $\partial R_2$ and since $u_2 \in L^2(0,T;C^0(\overline{R_2})) \cap  L^2(0,T;H^1(R_2))$, from
\cite{brezis} (see theorem IX.17 and remark 20) we obtain that $u_2 \in L^2(0,T;H^1_0(R_2))$.
We now use the heat equation for $u_2$ and obtain
\[\frac{d}{dt} \left(\int_{R_2} u_2^2\,dx\right) =- \int_{R_2}|\nabla u_2|^2\,dx \leq 0.\]
The initial condition satisfied by $u_2$ enables us to conclude that $u_2$ vanishes in $R_2 \times (0,T)$.
Whence, from the Holmgren's theorem again, we obtain that $u_2$ vanishes in $\Omega_2 \times (0,T)$, which is incompatible with the fact that
$(g_0,g_1) \neq 0$. This implies that $O_2 \subset O_1$ and we prove similarly that $O_1 \subset O_2$, which completes the proof. 
 \end{proof}
\subsection{Absence of initial condition}
It is a natural question to ask what is the role of the initial condition $u=0$ in $\Omega \times \{0\}$.
It is not difficult to see that, in the absence of initial condition in problem (\ref{obstacle}), uniqueness does not hold.
Indeed, let us consider for $d=2$ a domain $D$ that contains the square $S=(0,1) \times (0,1)$. Such square contains itself the square $\tilde{S}=(1/4,3/4) \times (1/4,3/4)$.
Let us now define
for $n \in \mathbb{N}$ 
\[\la_{n}=4n \pi \sqrt{2},\quad u_{n}(x_1,x_2,t)=e^{-\la_{n}^2 t}\sin(4n\pi x_1) \sin(4n\pi x_2).\]
The function $u_{1}$ is clearly a solution to the heat equation in $D$ that vanishes both on $\partial S$ and $\partial \tilde{S}$.
Hence it is a counterexample to uniqueness since two different obstacles are compatible with the same lateral Cauchy data. One could then expect that uniqueness is restored if we add extra lateral Cauchy data.
This is not true, since
all the functions $u_n$ are solutions to the heat equation in $D$, vanish both on $\partial S$ and $\partial \tilde{S}$ and
provide an infinite (still countable) number of lateral Cauchy data.
\subsection{Non zero initial condition}
In view of the above non uniqueness result when no information on the initial condition is given, it is another natural question whether we have uniqueness 
if we assume that the initial condition is $u=u_0$ in $\Omega \times \{0\}$, where $u_0$ is known but is not identically zero. To our best knowledge this question is open.
However, there exists at most one obstacle which is compatible with two pairs of Cauchy data $(g_0,g_1)$ and $(h_0,h_1)$ associated with two functions $u$ and $v$ with 
$(g_0,g_1) \neq (h_0,h_1)$. This result is simply obtained by applying the uniqueness theorem \ref{unique} to the function $u-v$, which satisfies problem (\ref{obstacle})
for a non vanishing pair of Cauchy data.
\subsection{Non negative initial condition}
We complete this review of uniqueness results by the special case of non negative initial condition and homogeneous Dirichlet data on the whole boundary $\partial D$.
The modified inverse obstacle problem we consider now consists, for some $u_0$ on $\Omega \times \{0\}$ and $g_1$ on $\Gamma \times (0,T)$, to find $O$ such that for some function 
$u$:
\begin{equation}
\label{obstacle_bis}
\left\{
\begin{array}{cccc}
& \partial_t u -\Delta u  = 0&  \text{in} &\Omega \times (0,T)\\
& u=0  & \text{on} & \partial D \times (0,T)\\
& \partial_\nu u=g_1 & \text{on} & \Gamma \times (0,T)\\
& u=0 & \text{on} &\partial O \times (0,T)\\
& u=u_0 & \text{on} &\Omega \times \{0\}.
\end{array}
\right.
\end{equation}
We have the following uniqueness result.
\begin{theorem}
\label{unique_bis}
Let us consider $u_0 \in L^2(\Omega)$.
For $i=1,2$, let two domains $O_i$ and corresponding functions $u_i \in L^2(0,T;H^1(\Omega_i)) \cap C^0([0,T];L^2(\Omega_i))$ satisfy problem (\ref{obstacle_bis}) with $g_1 \neq 0$ and $u_0 \geq 0$.
If we assume in addition that $u_i \in L^2(0,T;C^0(\overline{\Omega_i}))$ 
then we have $O_1=O_2$ and $u_1=u_2$.
\end{theorem}
\begin{proof}
We start the proof exactly the same way as in the proof of theorem \ref{unique} and we reuse the same notations. We hence have $u_1=u_2$ in $\tilde{\Omega}$.
Assume that 
$O_2$ is not contained in  $O_1$, which implies that there exists $\overline{x} \in \partial O_2$ and $\eps>0$ such that $B(\overline{x},\eps) \subset \Omega_1$, where 
$B(\overline{x},\eps)$ is the ball of center $\overline{x}$ and radius $\eps$. By continuity of $u_2$ up to the boundary of $\Omega_2$ and the fact that $u_2=0$ on $\partial O_2 \times (0,T)$, we have $u_1(\overline{x},\cdot)=0$ on $(0,T)$. Thanks to interior regularity of function $u_1$, such equality is pointwise, that is $u_1(\overline{x},t)=0$, for all $t \in (0,T)$.
Now let us use the mean value property for the heat equation (see \cite{evans}).
We define the heat ball $E(x,t)$ of radius $r$ in $\Omega \times (0,T)$ by
\[E(x,t)=\left\{(y,s) \in \Omega \times (0,T),\, \Phi(x-y,t-s) \geq \frac{1}{r^d}\right\},\]
with $\Phi$ the fundamental solution of the heat equation
\[\Phi(z,\tau)=\frac{H(\tau)}{(4\pi \tau)^{d/2}} e^{-|z|^2/4\tau},\]
and $H$ the Heavyside function.
A short computation proves that 
\be E(x,t)=\left\{(y,s) \in \Omega \times (0,T),\, y \in \overline{B(x,\alp(s))},\,  t- \frac{r^2}{4\pi}\leq s \leq t\right\},\label{hb}\ee
with 
\[\alp(s)=\sqrt{2d(t-s) \ln \left(\frac{r^2}{4\pi(t-s)}\right)}.\]
We note that for $(y,s) \in E(x,t)$, $\alp(s) \leq r \sqrt{d/2\pi}$.
As a result, for sufficiently small $r$, $E(\overline{x},t) \subset B(\overline{x},\eps) \times (0,T)$ for all $t \in (0,T)$.
The mean value property implies that
\[u_1(\overline{x},t)=\frac{1}{4r^d} \int_{E(\overline{x},t)}u_1(y,s) \frac{|y-\overline{x}|^2}{(s-t)^2}\,dyds.\] 
By the maximum principle for the heat equation (see \cite{brezis}, theorem X.3), we have that $u_1 \geq 0$ in $\Omega_1 \times (0,T)$. Therefore, that $u_1(\overline{x},t)=0$ for all $t \in (0,T)$ implies that
$u_1=0$ in $E(\overline{x},t)$ for all $t \in (0,T)$.
Now let us denote $2a=r^2/4\pi$ for sake of simplicity.
For any $t \in [2a,T-a/2]$ and $s \in [t-a,t-a/2] \subset [t-2a,t]$, $\alp(s) \geq \sqrt{d a \ln 2}$, which means in view of (\ref{hb}) that for $t \in [2a,T-a/2]$, the heat ball 
$E(\overline{x},t)$ contains the domain
$B(\overline{x},\sqrt{d a \ln 2}) \times (t-a,t-a/2)$.
By Holmgren's theorem, $u_1=0$ in $\Omega_1 \times (t-a,t-a/2)$ for all $t \in [2a,T-a/2]$, which implies that $u_1=0$ in $\Omega_1 \times (a,T-a)$.
Since the result is true for all sufficiently small $a$ we obtain that $u_1=0$ in $\Omega_1 \times (0,T)$, which contradicts
the fact that $g_1 \neq 0$. Then $O_2 \subset O_1$ and the reverse inclusion is obtained the same way.    
\end{proof}
\section{Some quasi-reversibility methods for the heat equation}
In this section we consider the heat equation with lateral Cauchy data and initial condition: for a pair of lateral Cauchy data $(g_0,g_1)$ on $\Gamma \times (0,T)$, find $u \in L^2(0,T;H^1(\Omega))$ such that
\begin{equation}
\label{lateral_initial}
\left\{
\begin{array}{cccc}
& \partial_t u -\Delta u  = 0&  \text{in} &\Omega \times (0,T)\\
& u=g_0  & \text{on} & \Gamma \times (0,T)\\
& \partial_\nu u=g_1 & \text{on} & \Gamma \times (0,T)\\
& u=0 & \text{on} &\Omega \times \{0\}.
\end{array}
\right.
\end{equation}
The problem (\ref{lateral_initial}) is well-known to be ill-posed, however
the Holmgren's theorem implies uniqueness of the solution $u$ with respect to the data $(g_0,g_1)$.
\subsection{A $H^1$-formulation}
We introduce the following open subsets of $\partial Q$: $\tilde{\Gamma}=\partial \Omega \setminus \overline{\Gamma}$,  $\Sigma=\Gamma \times (0,T)$, $\tilde{\Sigma}=\tilde{\Gamma} \times (0,T)$, $S_0=\Omega \times \{0\}$,
$S_T=\Omega \times \{T\}$ as well as the following functional sets. The space $H^{1/2}_{\tilde{\Sigma},S_T}(\Sigma)$ is the set of traces on $\Sigma$ of functions in $H^1(Q)$ which vanish on $\tilde{\Sigma}$ and on $S_T$. Its dual space is denoted 
$H^{-1/2}_{S_0}(\Sigma)$, which coincides with the set of restrictions to $\Sigma$ of distributions of $H^{-1/2}(\partial Q)$ the support of which is contained in $\overline{\Sigma \cup \tilde{\Sigma} \cup S_T}$.
Lastly, for $g_0 \in L^2(0,T;H^{1/2}(\Gamma))$, we set
\[H_g=\{u \in L^2(0,T;H^1(\Omega)),\,\,u|_{\Sigma}=g_0\}, \quad H_0=\{u \in L^2(0,T;H^1(\Omega)),\,\,u|_{\Sigma}=0\},\]
\[\tilde{V}_0=\{\la \in H^1(Q),\,\,\la|_{\tilde{\Sigma}}=0,\,\,\la|_{S_T}=0\}.\]
Due to Poincar\'e inequality, the spaces $H_0$ and $\tilde{V}_0$ can be endowed with the norms $(\int_Q |\nabla \cdot|^2\,dxdt)^{1/2}$ and $||\cdot||$, respectively, where $||\cdot||$ is defined by
\be ||\cdot||^2=\int_Q(\partial_t \cdot)^2\,dxdt + \int_Q |\nabla \cdot|^2\,dxdt\label{norm}\ee
and the corresponding scalar product is denoted by $((\cdot,\cdot))$.  
We will need the following lemma, which is a weak characterization of the solution to problem (\ref{lateral_initial}).
\begin{lemma}
\label{cha}
For $(g_0,g_1) \in L^2(0,T;H^{1/2}(\Gamma)) \times H^{-1/2}_{S_0}(\Sigma)$,
the function $u \in L^2(0,T;H^1(\Omega))$ is the solution to problem (\ref{lateral_initial}) if and only if $u \in H_{g}$ and for all $\mu \in \tilde{V}_0$,
\be 
-\int_Q u\,\partial_t \mu\,dxdt + \int_Q \nabla u \cdot \nabla \mu\,dxdt = \int_\Sigma g_1\,\mu\,dsdt,
\label{equi}
\ee
where the meaning of the last integral is duality between $H^{-1/2}_{S_0}(\Sigma)$ and $H^{1/2}_{\tilde{\Sigma},S_T}(\Sigma)$. 
\end{lemma}
\begin{proof}
To begin with, let us assume that $u \in H_g$ and satisfies the weak formulation (\ref{equi}). 
The definition of $V_g$ immediately implies that $u=g_0$ on $\Gamma \times (0,T)$.
By first choosing $\mu =\phi \in C^\infty_0(Q)$, we obtain $\partial_t u -\Delta u=0$ 
in $Q$ in the distributional sense. By using the definition (\ref{vector}), we conclude that ${\rm div}_{d+1} \UU=0$. 
In addition, from a classical integration by parts formula, we have for all $\mu \in H^1(Q)$,
\[\int_Q \UU \cdot \nabla_{d+1} \mu\,dX=- \int_Q {\rm div}_{d+1}\UU\,\mu\, dX + \int_{\partial Q} \UU\cdot \nu_{d+1}\,\mu\,dS,\]
where $\nabla_{d+1}=(\nabla,\partial_t)$, $X$ is the Lebesgue measure on $Q$ while $S$ is the corresponding surface measure on $\partial Q$ and the last integral has the meaning of duality between $H^{-1/2}(\partial Q)$ and $H^{1/2}(\partial Q)$.  
Now, for $\mu \in \tilde{V}_0$ and given that ${\rm div}_{d+1} \UU=0$, we obtain
\be \int_Q \UU \cdot \nabla_{d+1} \mu\,dX=\int_{\Sigma \cup S_0} \UU\cdot \nu_{d+1}\,\mu\,dS,\label{eq1}\ee
where the last integral has the meaning of duality between $H^{-1/2}(\Sigma \cup S_0)$ and $H^{1/2}_{00}(\Sigma \cup S_0)$, and $H^{1/2}_{00}(\Sigma \cup S_0)$ denotes the space of traces on $\Sigma \cup S_0$ of functions in $H^1(Q)$ which vanish on $\tilde{\Sigma} \cup S_T$.
The weak formulation (\ref{equi}) is equivalent to
\[\int_Q \UU \cdot \nabla_{d+1} \mu\,dX = \int_\Sigma g_1\,\mu\,dS, \quad \forall \mu \in \tilde{V}_0\]
that is to
\be \int_Q \UU \cdot \nabla_{d+1} \mu\,dX= \int_{\Sigma \cup S_0} \tilde{g}_1 \,\mu\,dS, \quad \forall \mu \in \tilde{V}_0, \label{eq2}\ee 
where $\tilde{g}_1$ is the extension by $0$ of $g_1$ to $\Sigma \cap S_0$ and the last integral has the meaning of duality between $H^{-1/2}(\Sigma \cup S_0)$ and $H^{1/2}_{00}(\Sigma \cup S_0)$. The distribution $\tilde{g}_1$ is well defined in $H^{-1/2}(\Sigma \cup S_0)$ from the fact that $g_1 \in H^{-1/2}_{S_0}(\Sigma)$.
Comparing equations (\ref{eq1}) and (\ref{eq2}) we end up with
$\UU\cdot \nu_{d+1}=\tilde{g}_1$ in the sense of $H^{-1/2}(\Sigma \cup S_0)$, which implies both $\partial_\nu u=g_1$ on $\Sigma$ and $u=0$ on $S_0$, that is $u$ solves problem (\ref{lateral_initial}).
We prove similarly that if $u \in L^2(0,T;H^1(\Omega))$ solves the problem (\ref{lateral_initial}) then it satisfies $u \in H_g$ and the weak formulation (\ref{equi}).
\end{proof}

The $H^1$-formulation of quasi-reversibility consists of the following problem for some real $\eps>0$:
for $(g_0,g_1) \in L^2(0,T;H^{1/2}(\Gamma)) \times H^{-1/2}_{S_0}(\Sigma)$, find $(u_{\eps},\la_{\eps}) \in H_g \times \tilde{V}_0$ such that
for all $(v,\mu) \in H_0 \times \tilde{V}_0$,
\be 
\left\{
\begin{array}{l}
\displaystyle 
-\int_Q v\, \partial_t \la_{\eps}\,dxdt + \int_Q \nabla v \cdot \nabla \la_{\eps}\,dxdt
+ \eps \int_Q \nabla u_{\eps} \cdot \nabla v\,dxdt=0,\\
\displaystyle
-\int_Q u_{\eps}\,\partial_t\mu\,dxdt - \int_Q \partial_t \la_{\eps}\,\partial_t \mu\,dxdt \\
\displaystyle  +\int_Q \nabla u_{\eps} \cdot \nabla \mu\,dxdt  
-\int_Q \nabla \la_{\eps} \cdot \nabla \mu\,dxdt= 
\int_{\Sigma }g_1\,\mu\,dsdt,
\end{array}\right.
\label{qr_lateral_initial}
\ee
where the last integral has the meaning of duality between $H^{-1/2}_{S_0}(\Sigma)$ and $H^{1/2}_{\tilde{\Sigma},S_T}(\Sigma)$. 
\begin{theorem}
\label{wellposed}
For any $(g_0,g_1) \in L^2(0,T;H^{1/2}(\Gamma)) \times H^{-1/2}_{S_0}(\Sigma)$, the problem (\ref{qr_lateral_initial}) has a unique solution $(u_{\eps},\la_{\eps})$ in $H_g \times \tilde{V}_0$.
Furthermore,
if there exists a (unique) solution $u \in L^2(0,T;H^1(\Omega))$ to problem (\ref{lateral_initial}) associated with data $(g_0,g_1)$,
then the solution $(u_{\eps},\la_{\eps})$ to problem (\ref{qr_lateral_initial}) associated with the same data $(g_0,g_1)$ satisfies
\[\lim_{\eps \rightarrow 0} u_{\eps}=u \quad {\rm in}\quad L^2(0,T;H^1(\Omega)),\quad 
\lim_{\eps \rightarrow 0} \la_{\eps}=0 \quad {\rm in} \quad H^{1}(Q).\]
\end{theorem}
\begin{proof}
We first prove well-posedness of the quasi-reversibility formulation.
Since we have assumed that $g_0 \in L^2(0,T;H^{1/2}(\Gamma))$, the set $H_g$ contains at least one element $\Phi$. In order to use the Lax-Milgram lemma, 
we define $\hat{u}_{\eps}=u_{\eps}-\Phi$ so that
the problem (\ref{qr_lateral_initial}) is equivalent to find $(\hat{u}_{\eps},\la_{\eps}) \in H_0 \times \tilde{V}_0$ such that for all $(v,\mu) \in H_0 \times \tilde{V}_0$, 
\[A((\hat{u}_{\eps},\la_{\eps}),(v,\mu))=L(v,\mu).\]
Here $L$ is a continuous linear form on $H_0 \times \tilde{V}_0$ (which we don't give explicitly as a function of $\Phi$), and $A$ is the continuous bilinear form on $H_0 \times \tilde{V}_0$ given by
\be
\left\{
\begin{array}{l}
\displaystyle 
A((\hat{u},\la),(v,\mu))= -\int_Q v\,\partial_t\la\,dxdt
+ \int_Q \nabla v \cdot \nabla \la\,dxdt + \int_Q \hat{u}\, \partial_t \mu\,dxdt \\
\displaystyle
- \int_Q \nabla \hat{u} \cdot \nabla \mu\,dxdt + \eps \int_Q \nabla \hat{u} \cdot \nabla v\,dxdt  
+\int_Q \partial_t \la \,\partial_t \mu\,dxdt   
+ \int_Q \nabla \la \cdot \nabla \mu\,dxdt.
\end{array}\right.
\label{a}
\ee
The Lax-Milgram lemma relies on the
coercivity of $A$ on the space $H_0 \times \tilde{V}_0$, which 
is straightforward from 
\be A((\hat{u},\la),(\hat{u},\la))=\eps\, ||\hat{u}||_{L^2(0,T;H^1(\Omega))}^2 + ||\la||^2.\label{coercif}\ee
Now let us prove the convergence of the quasi-reversibility solution to the exact solution for exact data. 
By the lemma \ref{cha}, the exact solution satisfies $u \in H_g$ and (\ref{equi}).
By subtracting the equation (\ref{equi}) to the second equation of problem (\ref{qr_lateral_initial}), we obtain that for all $(v,\mu) \in H_0 \times \tilde{V}_0$,
\be 
\left\{
\begin{array}{l}
\displaystyle 
-\int_Q v\,\partial_t\la_\eps\,dxdt+
\int_Q \nabla v \cdot \nabla \la_\eps\,dxdt
+ \eps \int_Q \nabla u_\eps \cdot \nabla v\,dxdt=0,\\
\displaystyle
-\int_Q (u_\eps-u)\,\partial_t\mu\,dxdt -\int_Q \partial_t \la_\eps\, \partial_t \mu\,dxdt\\
 \displaystyle + \int_Q \nabla (u_\eps-u) \cdot \nabla \mu\,dxdt  
- \int_Q \nabla \la_\eps \cdot  \nabla \mu\,dxdt= 0.
\end{array}\right.
\label{qr_mod}
\ee
We select the test functions $v$ and $\mu$ as $v=u_\eps-u \in H_0$ and $\mu=\la_\eps \in \tilde{V}_0$ in the above system (\ref{qr_mod}). By subtracting the two obtained equations, we end up with
\be
\eps \int_Q \nabla u_\eps \cdot \nabla (u_\eps-u)\,dxdt
+ \int_Q \left(\partial_t \la_\eps \right)^2\,dxdt 
+ \int_Q |\nabla \la_\eps|^2\,dxdt= 0,
\label{qr_modmod}
\ee
which can be simply rewritten as
\[\eps\, ||u_\eps||_{L^2(0,T;H^1(\Omega))}^2 + ||\la_\eps||^2=\eps\, ((u_\eps,u))_{L^2(0,T;H^1(\Omega))}.\]
It is then easy to derive that
\be ||u_\eps||_{L^2(0,T;H^1(\Omega))} \leq ||u||_{L^2(0,T;H^1(\Omega))},\quad ||\la_\eps|| \leq \sqrt{\eps}\, ||u||_{L^2(0,T;H^1(\Omega))}. \label{maj}\ee
From the first majoration (\ref{maj}) $u_\eps$ is bounded in $L^2(0,T;H^1(\Omega))$.
There exists a subsequence of $u_\eps$, still denoted $u_\eps$, that weakly converges to some $w\in L^2(0,T;H^1(\Omega))$, which happens
to belong to the set $H_g$ since such set is weakly closed.\\
Passing to the limit in the second equation of (\ref{qr_lateral_initial}) and using the second majoration (\ref{maj}), we obtain that for all $\mu \in \tilde{V}_0$,
\[
-\int_Q w\,\partial_t\mu\,dxdt + \int_Q \nabla w \cdot \nabla \mu\,dxdt = \int_\Sigma g_1\,\mu\,dsdt,
\]
that is $w$ solves problem (\ref{lateral_initial}) by using lemma \ref{cha} again. Uniqueness in problem (\ref{lateral_initial}) implies that $w=u$, so that $u_\eps$ weakly converges to $u$ in $L^2(0,T;H^1(\Omega))$.
The weak convergence and the first majoration (\ref{maj}) imply strong convergence.
A classical contradiction argument proves that the whole sequence $u_\eps$ (not only the subsequence) converges to $u$ in $L^2(0,T;H^1(\Omega))$.
\end{proof}
\begin{remark}
Our $H^1$-formulation of quasi-reversibility (\ref{qr_lateral_initial}) can be considered as an improvement of the formulation (13) proposed in \cite{becache_bourgeois_darde_franceschini} in the sense that the regularity of $u$ is only $L^2(0,T;H^1(\Omega))$ instead of $H^1(Q)$, which coincides with $L^2(0,T;H^1(\Omega)) \cap H^1(0,T;L^2(\Omega))$.
\end{remark}
\subsection{Some $H_{\rm div}$-formulations}
Another family of mixed quasi-reversibility methods can be proposed by rewriting problem (\ref{lateral_initial}) as
\begin{equation}
\label{lateral_initial_var}
\left\{
\begin{array}{cccc}
& \partial_t u - {\rm div}\, \pp  = 0&  \text{in} &\Omega \times (0,T)\\
& \nabla u  - \pp  = 0&  \text{in} &\Omega \times (0,T)\\
& u=g_0  & \text{on} & \Gamma \times (0,T)\\
& \pp\cdot \nu =g_1 & \text{on} & \Gamma \times (0,T)\\
& u=0 & \text{on} &\Omega \times \{0\}.
\end{array}
\right.
\end{equation}
We assume in this section that the exact solution $u$ satisfies the more restrictive condition $u \in H^1(Q)$.
\subsubsection{A basic formulation}
Let us introduce the space $H^{1/2}_{S_0}(\Sigma)$ of traces on $\Sigma$ of functions in $H^1(Q)$ which vanish on $S_0$. For $g_0 \in H^{1/2}_{S_0}(\Sigma)$ and $g_1 \in L^2(0,T;H^{-1/2}(\Gamma))$, we also consider the sets
\[V_g=\{u \in H^1(Q),\,\,u|_{\Sigma}=g_0,\,\,u|_{S_0}=0\},\quad V_0=\{u \in H^1(Q),\,\,u|_{\Sigma}=0\,\,u|_{S_0}=0\},\]
\[W_g=\{\pp \in L^2(0,T;H_{{\rm div},\Omega}),\,\pp\cdot \nu|_{\Sigma}=g_1\},\quad W_0=\{\pp \in L^2(0,T;H_{{\rm div},\Omega}),\,\pp\cdot \nu|_{\Sigma}=0\},\]
where $H_{{\rm div},\Omega}$ denotes the space of vector functions $\pp \in (L^2(\Omega))^d$ such that ${\rm div}\,\pp \in L^2(\Omega)$. The spaces $H_{{\rm div},\Omega}$ and $L^2(0,T;H_{{\rm div},\Omega})$ 
are naturally 
endowed with the norms $||\cdot||_{{\rm div},\Omega}$ and $||\cdot||_{\rm div}$ defined by
\[||\cdot||_{{\rm div},\Omega}^2=\int_\Omega \left(|\cdot|^2 + ({\rm div}\, \cdot)^2\right)dx,\quad 
||\cdot||_{\rm div}^2=\int_0^T ||\cdot||_{{\rm div},\Omega}^2\,dt,\]
respectively. 
The scalar product which corresponds to norm $||\cdot||_{\rm div}$ is denoted by $((\cdot,\cdot))_{\rm div}$. The space $V_0$ is endowed with the norm $||\cdot||$ already defined by (\ref{norm}).
In view of (\ref{lateral_initial_var}), another way to regularize problem (\ref{lateral_initial}) is for some real $\eps>0$:
for $(g_0,g_1) \in H^{1/2}_{S_0}(\Sigma) \times L^2(0,T;H^{-1/2}(\Gamma))$, find $(u_\eps,\pp_\eps) \in V_g \times W_g$ which minimizes
the functional
\[J_\eps(u,\pp)=\int_Q \left((\partial_t u - {\rm div}\, \pp)^2 + |\nabla u - \pp|^2\right)dxdt + \eps \int_Q \left((\partial_t u)^2 + |\nabla u|^2 + |\pp|^2 + ({\rm div}\, \pp)^2\right)dxdt.\]
The optimality for such minimization problem leads to the following mixed formulation for some $\eps>0$:
for $(g_0,g_1) \in H^{1/2}_{S_0}(\Sigma) \times L^2(0,T;H^{-1/2}(\Gamma))$, find $(u_\eps,\pp_\eps) \in V_g \times W_g$ such that for all $(v,\qq) \in V_0 \times W_0$
\be 
\left\{
\begin{array}{l}
\displaystyle \int_Q \left(\partial_t u_\eps-{\rm div}\, \pp_\eps)\partial_t v + (\nabla u_\eps - \pp_\eps)\cdot \nabla v\right)dxdt \\
\displaystyle 
+ \eps \int_Q \left(\partial_t u_\eps\,\partial_t v + \nabla u_\eps \cdot \nabla v\right)dxdt=0,
\\
\displaystyle
\int_Q \left(({\rm div}\, \pp_\eps-\partial_t u_\eps){\rm div}\,\qq + (\pp_\eps-\nabla u_\eps)\cdot \qq\right)dxdt \\
\displaystyle + \eps \int_Q \left(\pp_\eps\cdot \qq + ({\rm div}\, \pp_\eps)({\rm div}\, \qq)\right)dxdt=0.
\end{array}\right.
\label{qr_lateral_initial_var}
\ee
\subsubsection{A relaxed formulation}
In practice, the lateral Cauchy data $(g_0,g_1)$ on $\Gamma \times (0,T)$ are measurements and then are likely to be corrupted by noise. 
It is then tempting to modify the formulation (\ref{qr_lateral_initial_var}) as follows: on the one hand we assume that the data $g_0,g_1$ belong to $L^2(0,T;L^2(\Gamma))$, on the other hand we take them into account in a weak way rather than in a strong way. 
To this aim, the strong conditions $u=g_0$ and $\pp\cdot \nu=g_1$ on $\Gamma \times (0,T)$ have to be removed from the sets $V_g$ and $W_g$, respectively.
Moreover, since for a vector function $\pp$ in $L^2(0,T;H_{{\rm div},\Omega})$ the trace $\pp \cdot \nu$ on $\Gamma \times (0,T)$ is only in $L^2(0,T;H^{-1/2}(\Gamma))$ 
and not in $L^2(0,T;L^2(\Gamma))$ in general, we have to include such regularity assumption within the space of interest.
We will denote as $H_{{\rm div},\Omega,\Gamma}$ the space of vector functions $\pp \in H_{{\rm div},\Omega}$ such that $\pp \cdot \nu \in L^2(\Gamma)$. The spaces 
$H_{{\rm div},\Omega,\Gamma}$ and $L^2(0,T;H_{{\rm div},\Omega,\Gamma})$, 
are naturally 
endowed with the norms $||\cdot||_{{\rm div},\Omega,\Gamma}$ and $||\cdot||_{{\rm div},\Sigma}$ defined by
\[||\cdot||_{{\rm div},\Omega,\Gamma}^2=\int_\Omega \left(|\cdot|^2 + ({\rm div}\, \cdot)^2\right)dx + \int_\Gamma (\cdot\cdot \nu)^2\,ds ,\quad 
||\cdot||_{{\rm div},\Sigma}^2=\int_0^T ||\cdot||_{{\rm div},\Omega,\Gamma}^2\,dt,\]
respectively. 

By denoting 
$V=\{v \in H^1(Q),\,\,v|_{S_0}=0\}$ and $W=L^2(0,T;H_{{\rm div},\Omega,\Gamma})$,
we now consider 
the relaxed 
mixed formulation for some $\eps>0$:
for $(g_0,g_1) \in L^2(0,T;L^2(\Gamma)) \times L^2(0,T;L^2(\Gamma))$, find $(v_\eps,\qq_\eps) \in V \times W$ such that for all $(v,\qq) \in V \times W$
\be 
\left\{
\begin{array}{l}
\displaystyle \int_Q \left((\partial_t v_\eps-{\rm div}\, \qq_\eps)\partial_t v + (\nabla v_\eps - \qq_\eps)\cdot \nabla v\right)dxdt + \int_\Sigma v_\eps\, v\,dsdt\\
\displaystyle
+ \eps \int_Q \left(\partial_t v_\eps\,\partial_t v + \nabla v_\eps \cdot \nabla v\right)dxdt=  \int_\Sigma g_0\,v\,dsdt,
\\
\displaystyle
\int_Q \left(({\rm div}\, \qq_\eps-\partial_t v_\eps){\rm div}\,\qq + (\qq_\eps-\nabla v_\eps)\cdot \qq\right)dxdt + \int_\Sigma (\qq_\eps\cdot \nu) (\qq \cdot \nu)\,dsdt \\
\displaystyle +\eps \int_Q \left(\qq_\eps \cdot \qq + ({\rm div}\, \qq_\eps)({\rm div}\, \qq)\right)dxdt=\int_\Sigma g_1 (\qq \cdot \nu)\,dsdt. 
\end{array}\right.
\label{qr_lateral_initial_var2}
\ee
\begin{theorem}
\label{wellposed_var}
For any $(g_0,g_1) \in H^{1/2}_{S_0}(\Sigma) \times L^2(0,T;H^{-1/2}(\Gamma))$, the problem (\ref{qr_lateral_initial_var}) possesses a unique solution $(u_\eps,\pp_\eps)$ in $W_g \times V_g$. For any 
$(g_0,g_1) \in L^2(0,T;L^2(\Gamma)) \times L^2(0,T;L^2(\Gamma))$, the problem (\ref{qr_lateral_initial_var2}) possesses a unique solution $(v_\eps,\qq_\eps)$ in $W \times V$.
Furthermore,
if there exists a (unique) solution $u \in H^1(Q)$ to problem (\ref{lateral_initial}) associated with data $(g_0,g_1)$,
then the solution $(u_{\eps},\pp_{\eps})$ to problem (\ref{qr_lateral_initial_var}) and the solution $(v_{\eps},\qq_{\eps})$ to problem (\ref{qr_lateral_initial_var2}) associated with the same data $(g_0,g_1)$ satisfy
\[\lim_{\eps \rightarrow 0} u_{\eps},v_\eps=u \quad {\rm in}\quad H^1(Q),\quad 
\lim_{\eps \rightarrow 0} \pp_{\eps},\qq_\eps=\nabla u \quad {\rm in} \quad L^2(0,T;H_{{\rm div},\Omega}).\]
\end{theorem}
\begin{proof}
We prove the result in the case of formulation (\ref{qr_lateral_initial_var2}) for $(g_0,g_1) \in L^2(0,T;L^2(\Gamma)) \times L^2(0,T;L^2(\Gamma))$. The proof in the other case is almost the same.
Such formulation is equivalent to find $(v_\eps,\qq_\eps) \in V \times W$ such that for all $(v,\qq) \in V \times W$
\be B((v_{\eps},\qq_{\eps}),(v,\qq))=M(v,\qq), \label{b} \ee
where the bilinear form $B$ and the linear form $M$ are defined by
\[
\begin{array}{l}
\displaystyle B((u,\pp),(v,\qq))=\int_Q \left(\partial_t u-{\rm div}\, \pp)(\partial_t v -{\rm div}\,\qq)+ (\nabla u - \pp)\cdot (\nabla v- \qq)\right)dxdt \\
\displaystyle + \int_\Sigma u\, v\,dsdt
+ \int_\Sigma (\pp \cdot \nu) (\qq \cdot \nu)\,dsdt \\
\displaystyle
+ \eps \int_Q \left(\partial_t u\,\partial_t v + \nabla u \cdot \nabla v\right)dxdt + \eps \int_Q \left(\pp\cdot \qq + ({\rm div}\, \pp)({\rm div}\, \qq)\right)dxdt
\end{array}
\]
and 
\[M(v,\qq)=\int_\Sigma g_0\,v\,dsdt + \int_\Sigma g_1 (\qq \cdot \nu)\,dsdt.\]
We have
\be B((u,\pp),(u,\pp)) \geq {\rm min}(\eps,1) \left(||u||^2 + ||p||^2_{{\rm div},\Sigma}\right),\label{coerc}\ee
which proves the coercivity of $B$, so that well-posedness is a consequence of Lax-Milgram's lemma. Now let us prove the convergence result.
By using (\ref{lateral_initial_var}) we obtain that the exact solution satisfies $(u,\pp:=\nabla u) \in V \times W$ and for all $(v,\qq) \in V \times W$,
\be B_0((u,\pp),(v,\qq))=M(v,\qq), \label{b0}\ee
where $B_0$ coincides with $B$ for $\eps=0$.
We hence obtain, by subtracting (\ref{b0}) to (\ref{b}) and choosing $(v,\qq)=(v_\eps-u,\qq_\eps-\pp)$,
\be
\begin{array}{l}
\displaystyle \int_Q |\partial_t (v_\eps-u)-{\rm div}(\qq_\eps-\pp)|^2 dxdt+ \int_Q |\nabla (v_\eps-u) - (\qq_\eps-\pp)|^2dxdt \\
\displaystyle + \int_\Sigma (v_\eps-u)^2\,dsdt
+ \int_\Sigma ((\qq_\eps-\pp) \cdot \nu)^2\,dsdt 
+\eps\, ((v_\eps,v_\eps-u)) + \eps\, ((\qq_\eps,\qq_\eps-\pp))_{\rm div}=0.
\end{array}
\label{identity}
\ee
This identity implies that 
\be ((v_\eps,v_\eps-u)) + ((\qq_\eps,\qq_\eps-\pp))_{\rm div} \leq 0, \label{signe}\ee
so that $v_\eps$ is bounded in $V$ and $\qq_\eps$ is bounded in $L^2(0,T;H_{{\rm div},\Omega})$. Then there exists subsequences, still denoted $v_\eps$ and $\qq_\eps$, such that
$v_\eps \rightharpoonup w$ in $V$ and $\qq_\eps \rightharpoonup \rr$ in $L^2(0,T;H_{{\rm div},\Omega})$.
As a consequence, we have
$\partial_t v_\eps-{\rm div}\,\qq_\eps \rightharpoonup \partial_t w-{\rm div}\, \rr$ in $L^2(Q)$, $\nabla v_\eps-\qq_\eps \rightharpoonup \nabla w -\rr$ in $(L^2(Q))^d$, $v_\eps|_\Sigma \rightharpoonup w|_\Sigma$ in $L^2(0,T;H^{1/2}(\Gamma))$ and $\qq_\eps\cdot \nu|_\Sigma \rightharpoonup \rr \cdot \nu|_\Sigma$ in $L^2(0,T;H^{-1/2}(\Gamma))$.\\
The identity (\ref{identity}) also implies that $\partial_t v_\eps-{\rm div}\,\qq_\eps \rightarrow \partial_t u-{\rm div}\, \pp$ in $L^2(Q)$, $\nabla v_\eps-\qq_\eps \rightarrow \nabla u -\pp$ in $(L^2(Q))^d$, $v_\eps|_\Sigma \rightarrow u|_\Sigma$ in $L^2(\Sigma)$ and $\qq_\eps\cdot \nu|_\Sigma \rightarrow \pp \cdot \nu|_\Sigma$ in $L^2(\Sigma)$.
We conclude that $\rr \cdot \nu|_\Sigma \in L^2(\Sigma)$, and that $(w,\rr) \in V \times W$ satisfies problem (\ref{lateral_initial_var}).
By uniqueness we have $(w,\rr)=(u,\pp)$, so that 
$v_\eps \rightharpoonup u$ in $V$ and $\qq_\eps \rightharpoonup \pp$ in $L^2(0,T;H_{{\rm div},\Omega})$.
Strong convergence of $(v_\eps,\qq_\eps)$ to $(u,\pp)$ in $V \times L^2(0,T;H_{{\rm div},\Omega})$ is again a consequence of weak convergence from (\ref{signe}), and we complete the proof
as in the proof of theorem \ref{wellposed}.
\end{proof}
\begin{remark} On could regret that in the relaxed formulation (\ref{qr_lateral_initial_var2}) the Neumann data has to be in $L^2(0,T;L^2(\Gamma))$ instead of $L^2(0,T;H^{-1/2}(\Gamma))$.
As done in \cite{darde} in the elliptic case, a possible strategy to cope with the less regular case $g_1 \in L^2(0,T;H^{-1/2}(\Gamma))$ in (\ref{lateral_initial}) is to introduce a lifting of the Neumann data, for example by solving the forward problem 
 \begin{equation}
\label{U}
\left\{
\begin{array}{cccc}
& \partial_t U -\Delta U = 0&  \text{in} &\Omega \times (0,T)\\
& \partial_\nu U=\tilde{g}_1  & \text{on} & \partial \Omega \times (0,T)\\
& U=0 & \text{on} &\Omega \times \{0\},
\end{array}
\right.
\end{equation}
where $\tilde{g}_1 \in L^2(0,T;H^{-1/2}(\partial \Omega))$ is an extension of $g_1$ to $\partial \Omega \times (0,T)$. From
Theorem X.9 in \cite{brezis}, we obtain that problem (\ref{U}) is well-posed in $L^2(0,T;H^1(\Omega)) \cap C^0([0,T];L^2(\Omega))$.
Then the function $\hat{u}=u-U$, where $u$ and $U$ satisfy (\ref{lateral_initial}) and (\ref{U}), respectively, satisfies
\[
\left\{
\begin{array}{cccc}
& \partial_t \hat{u} -\Delta \hat{u}  = 0&  \text{in} &\Omega \times (0,T)\\
& \hat{u}=g_0-U  & \text{on} & \Gamma \times (0,T)\\
& \partial_\nu \hat{u}= 0& \text{on} & \Gamma \times (0,T)\\
& \hat{u}=0 & \text{on} &\Omega \times \{0\},
\end{array}
\right.
\]
with $g_0-U|_{\Gamma \times (0,T)} \in L^2(0,T;H^{1/2}(\Gamma)) \subset L^2(0,T;L^2(\Gamma))$.
We are now in a position to apply the relaxed formulation (\ref{qr_lateral_initial_var2}).
\end{remark}
\subsubsection{An iterated formulation}
Following the idea of \cite{darde} a refinement of the relaxed formulation (\ref{qr_lateral_initial_var2}) consists in iterating it. For  some $\eps>0$
and $(g_0,g_1) \in L^2(0,T;L^2(\Gamma)) \times L^2(0,T;L^2(\Gamma))$, we set $(v^{-1}_\eps,\qq^{-1}_\eps)=(0,0)$ and for all $M \in \mathbb{N}$, we define $(v^{M}_\eps,\qq^{M}_\eps) \in V \times W$ such that for all $(v,\qq) \in V \times W$
\be 
\left\{
\begin{array}{l}
\displaystyle \int_Q \left((\partial_t v^M_\eps-{\rm div}\, \qq^M_\eps)\partial_t v + (\nabla v^M_\eps - \qq^M_\eps)\cdot \nabla v\right)dxdt + \int_\Sigma v^M_\eps\, v\,dsdt\\
\displaystyle
+ \eps \int_Q \left(\partial_t v^M_\eps\,\partial_t v + \nabla v^M_\eps \cdot \nabla v\right)dxdt\\
\displaystyle =  \int_\Sigma g_0\,v\,dsdt + \eps \int_Q \left(\partial_t v^{M-1}_\eps\,\partial_t v + \nabla v^{M-1}_\eps \cdot \nabla v\right)dxdt,
\\
\displaystyle
\int_Q \left(({\rm div}\, \qq^M_\eps-\partial_t v^M_\eps){\rm div}\,\qq + (\qq^M_\eps-\nabla v^M_\eps)\cdot \qq\right)dxdt + \int_\Sigma (\qq^M_\eps\cdot \nu) (\qq \cdot \nu)\,dsdt \\
\displaystyle +\eps \int_Q \left(\qq^M_\eps \cdot \qq + ({\rm div}\, \qq^M_\eps)({\rm div}\, \qq)\right)dxdt\\
\displaystyle =\int_\Sigma g_1 (\qq \cdot \nu)\,dsdt + \eps \int_Q \left(\qq^{M-1}_\eps \cdot \qq + ({\rm div}\, \qq^{M-1}_\eps)({\rm div}\, \qq)\right)dxdt. 
\end{array}\right.
\label{qr_lateral_initial_it}
\ee
From \cite{darde} (see the abstract theorem 4.2), we have the following theorem, which encourages us to simultaneously choose $\eps$ small and $M$ large.
\begin{theorem}
\label{wellposed_it}
For any 
$(g_0,g_1) \in L^2(0,T;L^2(\Gamma)) \times L^2(0,T;L^2(\Gamma))$ and any $\eps>0$ and $M \in \mathbb{N}$, the problem (\ref{qr_lateral_initial_it}) possesses a unique solution $(v^M_\eps,\qq^M_\eps)$ in $W \times V$.
Furthermore,
if there exists a (unique) solution $u \in H^1(Q)$ to problem (\ref{lateral_initial}) associated with data $(g_0,g_1)$,
then the solution $(v^M_{\eps},\qq^M_{\eps})$ to problem (\ref{qr_lateral_initial_it}) associated with data $(g_0,g_1)$ satisfies for all fixed $M \in \mathbb{N}$
\[\lim_{\eps \rightarrow 0} v^M_\eps=u \quad {\rm in}\quad H^1(Q),\quad 
\lim_{\eps \rightarrow 0} \qq^M_\eps=\nabla u \quad {\rm in} \quad L^2(0,T;H_{{\rm div},\Omega})\]
and for all fixed $\eps>0$
\[\lim_{M \rightarrow + \infty} v^M_\eps=u \quad {\rm in}\quad H^1(Q),\quad 
\lim_{M \rightarrow +\infty} \qq^M_\eps=\nabla u \quad {\rm in} \quad L^2(0,T;H_{{\rm div},\Omega}).\]
\end{theorem}
\begin{remark}
As emphasized in \cite{darde}, an advantage of the iterated formulation is that the convergence with respect to $M$ is achieved for any $\eps>0$. Hence it enables us
to choose $\eps$ not too small in problem (\ref{qr_lateral_initial_it}), which as a result will be not too ill-posed.
\end{remark}
\section{The ``exterior approach"}
Let us consider, for some data $(g_0,g_1)$ on $\Gamma \times (0,T)$, an obstacle $O$ and an associated function $u \in L^2(0,T;H^1(\Omega))$ which satisfy problem
(\ref{obstacle}). Our aim is, following the idea first introduced in \cite{bourgeois_darde1}, to define with the help of $u$ a decreasing sequence of open domains $O_n$ 
which converge in the sense of Hausdorff distance (for open domains) to the actual obstacle $O$.
The notion of Hausdorff distance for open domains is for example defined in \cite{henrot_pierre}.
In practice, of course, the true solution $u$ cannot be used to identify our obstacle $O$ since it cannot be computed from our data $(g_0,g_1)$.
However, as it was seen in the previous section, the true solution $u$ can be approximated with the help of some quasi-reversibility formulation.
This is the basic idea of the ``exterior approach". 

Let us consider a function $V \in H^1(D)$ such that
\be
\left\{
\begin{array}{rcll}
\displaystyle V & = & \displaystyle \sqrt{\int_0^T (u(\cdot,t))^2\,dt}  & \displaystyle \mbox{ \rm{in} }\Omega\\
\displaystyle V & \displaystyle \leq & 0 & \displaystyle \mbox{ \rm{in} } O.
\end{array}\right.
\label{V}
\ee
Let us verify that such a function exists.
\begin{lemma}
The function $V$ which satisfies the first equality of (\ref{V}) and $V=0$ in $O$ belongs to $H^1(D)$.
\end{lemma}
\begin{proof}
First of all, since $u=0$ on $\partial O \times (0,T)$, we obtain that $V$ is continuous across the boundary $\partial O$ of $O$.
As a consequence, it suffices to prove that $V|_\Omega \in H^1(\Omega)$. It is readily seen that $V|_\Omega \in L^2(\Omega)$.
Now, for $i=1,...,d$ and $x \in \Omega$,
\[\displaystyle \left(\deri{V}{x_i}(x)\right)^2= \frac{1}{\int_0^T (u(x,t))^2\,dt} \left(\int_0^T u(x,t)\deri{u}{x_i}(x,t)\,dt\right)^2 \leq \int_0^T \left(\deri{u}{x_i}(x,t)\right)^2\,dt.\]
That $u \in L^2(0,T;H^1(\Omega))$ implies that $\partial V/\partial x_i \in L^2(\Omega)$ for all $i$, which completes the proof.
\end{proof}

Let us choose $f \in H^{-1}(D)$ such that in the sense of
$H^{-1}(D)$,
\be f -\Delta V \geq 0\label{source}.\ee
For some open domain $\omega \subset D$ and $g \in H^{-1}(D)$, let us define by $v_{g,\omega}$ the solution $v \in H^1_0(\omega)$ of the Poisson problem $\Delta v=g$.
We now define a sequence of open domains $O_n$ by following induction. 
We first consider 
an open domain $O_0$ such that $O \subset O_0 \Subset D$.
The open domain $O_n$ being given, we define
\be O_{n+1}=O_n \setminus {\rm supp}(\sup(\phi_n,0)),\label{recurrence}\ee
where $\phi_n=V+v_{g,O_n}$ and $g:=f -\Delta V$ (${\rm supp}$ denotes the support of a function). 
Equivalently, if the open domain $O_n$ is Lipschitz smooth, the function $\phi_n$ is defined as the unique solution in $H^1(O_n)$ of the boundary value problem
 \be
\left\{
\begin{array}{ccc}
\Delta \phi_n = f & {\rm in} &O_n\\
 \phi_n=V & {\rm on} & \partial O_n.
\end{array}\right.
\label{poisson}
\ee
Since the open domains $O_n$ form a decreasing sequence, we know from \cite{henrot_pierre} that such sequence converges, in the sense of Hausdorff distance for open domains,
to some open domain $O_\infty$. Besides, the definition of the open sets $O_n$ implies that they all contain the obstacle $O$. Therefore it can be seen from (\ref{poisson}) that $\phi_n$ only depends on the fixed distribution $f$ and on the values of $u$ on the boundary $\partial O_n$. Furthermore $O \subset O_\infty$.
The convergence of the sequence of open domains $O_n$ to the actual obstacle $O$ (in other words $O_\infty=O$) is given by the following theorem, 
provided we assume convergence of the functions $v_{g,O_n}$ with respect to the domain.
\begin{theorem}
\label{main}
Let us consider a Lipschitz domain $O$ and a function $u \in L^2(0,T;H^1(\Omega))$ which satisfy problem (\ref{obstacle}).
Let us choose some $V \in H^1(D)$ and $f \in H^{-1}(D)$ which satisfy (\ref{V}) and (\ref{source}), respectively, and let us denote $g=f -\Delta V$.
Now we choose an open domain $O_0$ such that $O \subset O_0 \Subset D$ and consider the decreasing sequence of open domains $O_n$ defined by (\ref{recurrence}).
Let us denote by $O_\infty$ the limit of the sequence $(O_n)$ in the sense of Hausdorff distance for open domains.\\
If we assume that the sequence of functions $v_{g,O_n}$ converge in $H^1_0(D)$ to the function $v_{g,O_\infty}$, then $O_\infty=O$.
\end{theorem}

We omit the proof of theorem \ref{main} since it is very close to that of theorem 2.5 in \cite{bourgeois_darde1}. The only difference in the proof concerns the unique continuation argument which
is employed: it was related to the Laplace equation in \cite{bourgeois_darde1} while it is related to the heat equation in the present paper. 
\begin{remark}
In \cite{henrot_pierre} (see also \cite{bourgeois_darde1}), several situations in which the sequence of functions $v_{g,O_n}$ converge in $H^1_0(D)$ 
to the function $v_{g,O_\infty}$ are analyzed.
For $d=2$, it suffices that all the open domains $D \setminus \overline{O_n}$, $n \in \mathbb{N}$, be connected.
For arbitrary $d \geq 2$, it suffices that all the open domain $O_n$, $n \in \mathbb{N}$, be uniformly Lipschitz with respect to $n$.
\end{remark}

The theorem \ref{main} suggests the following ``exterior approach" algorithm:

\vspace{0.2cm}
{\bf Algorithm}

\vspace{0.1cm}
\begin{enumerate}
\item Choose an initial guess $O_0$ such that $O \subset O_0 \Subset D$.
\item Step 1: for a given $O_n$, compute some quasi-reversibility solution $u_n$ in $\Omega_n \times (0,T)$, where  
$\Omega_n:=D \setminus \overline{O_n}$.
\item Step 2: for a given $u_n$ in $\Omega_n \times (0,T)$, compute $V_n(x)=||u_n(x,\cdot)||_{L^2(0,T)}$ in $\Omega_n$ and the solution $\phi_n$ in $O_n$ of the Poisson problem 
\be 
\left\{
\begin{array}{ccc}
\Delta \phi_n = f & {\rm in} &O_n\\
 \phi_n=V_n & {\rm on} & \partial O_n
\end{array}\right.
\label{poisson_n}
\ee
for sufficiently large $f$.
Compute $O_{n+1}=\{x \in O_n,\,\,\phi_n(x)<0\}$.
\item Go back to step 1 until some stopping criterion is satisfied.
\end{enumerate}
\section{Some numerical applications}
\subsection{A tensorized finite element method}
\label{sec_tensor}
Our numerical experiments will be based on the iterated mixed formulation (\ref{qr_lateral_initial_it}).
However the discretization is presented for $M=0$ for sake of simplicity, which corresponds to formulation (\ref{qr_lateral_initial_var2}). 
In this section we restrict ourselves to a polygonal domain $\Omega$ in $\mathbb{R}^2$.
We consider a family of triangulations $T_h$ of the domain $\overline{\Omega}$ such that the diameter of each triangle $K \in T_h$ is bounded by $h>0$ and such that $T_h$ is regular in the sense of \cite{ciarlet}. We also introduce a subdivision $I_h$ of the domain $[0,T]$ such that the size of each interval $L \in I_h$ is also bounded by $h$.
We assume that $\overline{\Gamma}$ is formed by the union of edges of some triangles of $T_h$.
As can be seen in the algorithm of the ``exterior approach" above, we have to compute several integrals of the solutions of the quasi-reversibility method over $t \in [0,T]$.
For practical reasons we are hence tempted to use some tensorized finite elements to discretize the spaces $V$ and $W$. More precisely,
the discretized space $V_h \subset V$ in the domain $\Omega \times (0,T)$ is the set denoted $V_{h,\Omega} \otimes V_{h,T}$ formed by the linear combination of standard products of all basis functions of the discretized space $V_{h,\Omega}$ in $\Omega$ by all basis functions of the discretized space $V_{h,T}$ in $(0,T)$, where: 
$V_{h,\Omega}$ consists of the standard $P^1$ triangular finite element on $\Omega$ and $V_{h,T}$ 
consists of the $P^1$ finite element on $(0,T)$. Similarly, the discretized space $W_h \subset W$ in the domain $\Omega \times (0,T)$ is the set $W_{h,\Omega} \otimes W_{h,T}$, where: $W_{h,\Omega}$ 
consists of the Raviart-Thomas finite element $RT^0$ on $\Omega$ and $V_{h,T}$ 
consists of the standard $P^0$ finite element on $(0,T)$.
The finite elements $P^k$ and $RT^k$ for $k \in \mathbb{N}$ are for example described in \cite{brezzi_fortin}, 
the finite element $RT^k$ being introduced in \cite{raviart_thomas}. Since the maximal diameter $h$ of the two-dimensional 
triangular mesh of $\Omega$ coincides with that of the subdivision of $(0,T)$ and since both finite elements in space and time provide a linear error estimate
with respect to $h$, we expect that the resulting tensorized finite element will also provide a linear error estimate
with respect to $h$.
For each triangle $K \in T_h$ or interval $L \in I_h$ and $k \in \mathbb{N}$, $P_k(K)$ and $P_k(L)$ denote the space of polynomial functions of degree lower or equal to $k$.
The spaces $V_{h,\Omega}$, $V_{h,T}$, $W_{h,\Omega}$ and $W_{h,T}$ are defined as follows:
\[V_{h,\Omega} = \{f_h \in H^1(\Omega),\,\, f_h|_K \in P_1(K),\,\forall K \in T_h \},\]
\[V_{h,T} = \{\phi_h \in H^1(0,T),\,\, \phi_h|_L \in P_1(L),\,\forall L \in I_h, \quad \phi_h(0)=0\}, \]
\[W_{h,\Omega}= \{\ff_h \in (L^2(\Omega))^2,\,{\rm div}\ff_h \in L^2(\Omega),\,\,  \ff_h \in (P_0(K))^2 + x P_0(K),\,\forall K \in T_h\},\]
\[W_{h,T} = \{\phi_h \in L^2(0,T),\,\,\phi_h|_L \in P_0(L),\,\forall L \in I_h \},\]
where $x \in \mathbb{R}^2$ is the spatial coordinate.
We recall that the trace of some $f_h \in V_{h,\Omega}$ or $\phi_h \in V_{h,T}$ is continuous across the intersection of two elements, 
while for some $\ff_h \in W_{h,\Omega}$, the trace of $\ff_h \cdot \nu$ is continuous across such intersection, where $\nu$ is the corresponding normal vector.

The discretized formulation of (\ref{qr_lateral_initial_var2}) for $\eps,h>0$ is the following: 
for $(g_0,g_1) \in L^2(0,T;L^2(\Gamma)) \times L^2(0,T;L^2(\Gamma))$, find $(v_{\eps,h},\qq_{\eps,h}) \in V_h \times W_h$ such that for all $(v_h,\qq_h) \in V_h \times W_h$
\be 
\left\{
\begin{array}{l}
\displaystyle \int_Q \left((\partial_t v_{\eps,h}-{\rm div}\, \qq_{\eps,h})\partial_t v_h + (\nabla v_{\eps,h} - \qq_{\eps,h})\cdot \nabla v_h\right)dxdt + \int_\Sigma v_{\eps,h}\, v_h\,dsdt\\
\displaystyle
+ \eps \int_Q \left(\partial_t v_{\eps,h}\,\partial_t v_h + \nabla v_{\eps,h} \cdot \nabla v_h\right)dxdt=  \int_\Sigma g_0\,v_h\,dsdt,
\\
\displaystyle
\int_Q \left(({\rm div}\, \qq_{\eps,h}-\partial_t v_{\eps,h}){\rm div}\,\qq_h + (\qq_{\eps,h}-\nabla v_{\eps,h})\cdot \qq_h\right)dxdt + \int_\Sigma (\qq_{\eps,h}\cdot \nu) (\qq_h \cdot \nu)\,dsdt \\
\displaystyle +\eps \int_Q \left(\qq_{\eps,h} \cdot \qq_h + ({\rm div}\, \qq_{\eps,h})({\rm div}\, \qq_h)\right)dxdt=\int_\Sigma g_1 (\qq_h \cdot \nu)\,dsdt. 
\end{array}\right.
\label{qrh_lateral_initial_var2}
\ee
The error estimate due to the discretization, that is the discrepancy between the solution to problem (\ref{qrh_lateral_initial_var2}) and the solution to problem (\ref{qr_lateral_initial_var2}), is given by the following theorem.
\begin{theorem}
\label{th1d}
For all $\eps,h>0$, the problem (\ref{qrh_lateral_initial_var2}) has a unique solution $(v_{\eps,h},\qq_{\eps,h}) \in V_h \times W_h$. Furthermore, if $\eps\leq 1$ and $(v_{\eps},\qq_{\eps})$ belongs to $H^2(Q) \times \{H^1(0,T;(H^1(\Omega))^2) \cap L^2(0,T;(H^2(\Omega))^2)\}$, then
\[||v_{\eps,h}-v_{\eps}||+ ||\qq_{\eps,h}-\qq_{\eps}||_{{\rm div}}
\leq C\,\frac{h}{\sqrt{\eps}}\left(||v_{\eps}||_{H^{2}(Q)} + ||\qq_{\eps}||_{H^1(0,T;(H^1(\Omega))^2) \cap L^2(0,T;(H^{2}(\Omega))^2)} \right),\]
where
$C>0$ is independent of $\eps$ and $h$.
\end{theorem}

In order to prove theorem \ref{th1d}, we first need to prove the following lemma which specifies the interpolation error provided by the tensor product of two finite elements,
from the knowledge of the interpolation error for each one. 
\begin{lemma}
\label{tensor}
Let us consider two Hilbert spaces $F \subset H$ and a family of discretization subspaces $H_{h,1} \subset H$ depending on $h$ such that for all 
$f \in F$,
\be ||f -\pi_{h,1}\,f||_H \leq c_1\,h ||f||_{F}, \label{F}\ee
where $\pi_{h,1}\,f$ is the orthogonal projection of $f$ onto $H_{h,1}$ in $H$.
We also consider a family of discretization subspaces $H_{h,2} \subset L^2(0,T)$ depending on $h$ such that for all 
$\phi \in H^1(0,T;H)$,
\be ||\phi -\pi_{h,2}\,\phi||_{L^2(0,T;H)} \leq c_2 \,h \left\|\deri{\phi}{t}\right\|_{L^2(0,T;H)},\label{phi}\ee
where $\pi_{h,2}\,\phi$ is the orthogonal projection of $\phi$ onto $H_{h,2}$ in $L^2(0,T)$.

Then for $p \in L^2(0,T;F) \cap H^1(0,T;H)$, we have
\[||p -\pi_{h}\,p||_{L^2(0,T;H)} \leq c\,h ||p||_{L^2(0,T;F) \cap H^1(0,T;H)},\]
where $\pi_h$ is the orthogonal projection of $p$ onto the tensor product $H_{h,1} \otimes H_{h,2}$ in $L^2(0,T;H)$.
\end{lemma}  
\begin{proof}
Since $\pi_{h,2}(\pi_{h,1}\,p) \in H_{h,1} \otimes H_{h,2}$,
we have
\[||p-\pi_h\,p||_{L^2(0,T;H)}=\inf_{p_h \in H_{h,1} \otimes H_{h,2}} ||p-p_h||_{L^2(0,T;H)} \leq ||p-\pi_{h,2}(\pi_{h,1}\,p)||_{L^2(0,T;H)}.\]
Hence
\[||p -\pi_{h}\, p||_{L^2(0,T;H)} \leq ||p -\pi_{h,1}\, p||_{L^2(0,T;H)} + ||\pi_{h,1}\,p -\pi_{h,2}(\pi_{h,1}\,p)||_{L^2(0,T;H)}\]
As for the first term, we have
\[ ||p -\pi_{h,1}\,p||^2_{L^2(0,T;H)}=\int_0^T ||p-\pi_{h,1}\,p||_H^2\,dt \leq  c^2_1\,h^2 \int_0^T ||p(.,t)||^2_{F}\,dt\]
by using (\ref{F}) for $f=p(.,t)$. We end up with
\[ ||p -\pi_{h,1}\,p||_{L^2(0,T;H)} \leq  c_1\,h||p||_{L^2(0,T;F)}.\]
As for the second term, we have
\[||\pi_{h,1}\,p -\pi_{h,2}(\pi_{h,1}\,p)||_{L^2(0,T;H)} \leq c_2 \,h \left\|\deri{(\pi_{h,1}\,p)}{t}\right\|_{L^2(0,T;H)},\]
with the help of (\ref{phi}) for $\phi=\pi_{h,1}\,p$.
By using the fact that $\partial(\pi_{h,1}\,p)/\partial t=\pi_{h,1}(\partial p/\partial t)$ and the boundedness of the operator $\pi_{h,1}$ from $H$ to itself with a bound $\leq 1$,
we obtain
\[||\pi_{h,1}\,p -\pi_{h,2}(\pi_{h,1}\,p)||_{L^2(0,T;H)} \leq c_2 \,h \left\|\deri{p}{t}\right\|_{L^2(0,T;H)} \leq c_2 \,h ||p||_{H^1(0,T;H)},\]
which completes the proof.
\end{proof}
\begin{proof}[Theorem \ref{th1d}]
Well-posedness of problem (\ref{qrh_lateral_initial_var2}) is based on the same arguments as in the proof of theorem \ref{wellposed_var}.
By denoting $X_{\eps}=(v_{\eps},\qq_{\eps}) \in V \times W$ and
$X_{\eps,h}=(v_{\eps,h},\qq_{\eps,h}) \in V_h \times W_h$,
and since the bilinear form $B$ is symmetric, 
the solution $X_{\eps,h}$ minimizes the functional
$B(X_\eps-Y_h,X_\eps-Y_h)$ over all $Y_h \in V_h \times W_h$.
With the help of (\ref{coerc}) and the fact that $\eps \leq 1$, we obtain a constant $c>0$ such that
\[||X_{\eps}-X_{\eps,h}||_{V \times W} \leq  \frac{c}{\sqrt{\eps}}\, \inf_{Y_h \in V_h \times W_h} ||X_{\eps}-Y_h||_{V \times W}.\]
We have now to estimate the interpolation errors $v_{\eps}-\pi_h v_\eps \in V$  and $\qq_{\eps}-\pi_h \qq_{\eps} \in W$,
where $\pi_h$ is the generic projection operator of an element which belongs to an infinite dimensional space onto the corresponding finite element space.
The first error is simple to obtain since the finite element defined as the tensor product of the $P^1$ element in $2D$ by the $P^1$ element in $1D$ coincides with the prismatic finite element in 3D (see for example \cite{ciarlet}).
The error estimate for such $3D$ finite element is well-known and we have for $v_\eps \in H^2(Q)$,
\[||v_{\eps}-\pi_h v_\eps|| \leq C\,h ||v_\eps||_{H^{2}(Q)}.\]
for some constant $C>0$.
Let us consider the second error.
For $\ff\in (H^1(\Omega))^2$ such that ${\rm div}\ff \in H^1(\Omega)$, we have the inequality (see for example \cite{brezzi_fortin,raviart_thomas})
\be ||\ff-\pi_h \ff||_{H_{\rm div,\Omega}} \leq C'_\Omega\,h (||\ff||_{(H^1(\Omega))^2}+||{\rm div} \ff||_{H^1(\Omega)}),\label{estim1}\ee
where $\pi_h \ff$ is the interpolate of $\ff$ on $W_{h,\Omega}$.
Given the definition of space $W$ we also need to estimate the trace of $(\ff-\pi_h \ff) \cdot \nu$ on $\Gamma$. 
If in addition we assume that $\ff \in (H^{2}(\Omega))^2$, then the trace of $\ff \cdot \nu$ on each segment $\Gamma_K$ of $\Gamma$ belongs to $H^1(\Gamma_K)$, and the interpolation error on each such segment
amounts, since $(\pi_h \ff) \cdot \nu$ is the mean value of $\ff \cdot \nu$ on each segment (see \cite{brezzi_fortin,raviart_thomas}), to an interpolation of a $H^1$ function on a 
segment by a constant.
We hence have, by denoting $\tau$ the curvilinear abscissa,
\[ ||(\ff-\pi_h \ff) \cdot \nu ||^2_{L^2(\Gamma)}= \sum_{\Gamma_K \subset \overline{\Gamma}} \int_{\Gamma_K }|(\ff-\pi_h \ff) \cdot \nu|^2\,d \tau \]
\[\leq  \sum_{\Gamma_K \subset \overline{\Gamma}}\,c^2\,h^2 \int_{\Gamma_K }\left|\deri{(\ff \cdot \nu)}{\tau}\right|^2\,d \tau =c^2\,h^2\sum_{\Gamma_K \subset \overline{\Gamma}}\int_{\Gamma_K }\left|\deri{\ff}{\tau} \cdot \nu\right|^2\,d \tau \leq c^2\,h^2 ||\ff||^2_{(H^1(\Gamma))^2}.\]
By using the continuity of the trace from $H^2(\Omega)$ to $H^1(\Gamma)$, we obtain
\be ||(\ff-\pi_h \ff) \cdot \nu ||_{L^2(\Gamma)} \leq C''_\Omega\,h ||\ff||_{(H^2(\Omega))^2} \label{estim2}.\ee
Gathering the estimates (\ref{estim1}) and (\ref{estim2}) we obtain that
\[ ||\ff-\pi_h \ff||_{H_{\rm div,\Omega,\Gamma}} \leq C_\Omega\, h ||\ff||_{(H^2(\Omega))^2}.\]
We also have, for $H=H_{{\rm div},\Omega,\Gamma}$
and $\phi \in H^1(0,T;H)$,  
\[||\phi-\pi_h \phi||_{L^2(0,T;H)} \leq C_T\,h\left\|\deri{\phi}{t}\right\|_{L^2(0,T;H)},\]
where $\pi_h \phi$ is the interpolate of $\phi$ on $W_{h,T}$.
By applying lemma \ref{tensor} with $H=H_{{\rm div},\Omega,\Gamma}$, $F=(H^2(\Omega))^2$, $H_{h,1}=W_{h,\Omega}$ and $H_{h,2}=W_{h,T}$, we obtain that there exists a constant $C$ such that
\[||\qq_{\eps}-\pi_h \qq_\eps||_{{\rm div},\Sigma} \leq C\,h ||\qq_\eps||_{H^1(0,T;H_{{\rm div},\Omega,\Gamma}) \cap L^2(0,T;(H^{2}(\Omega))^2)}.\] 
The continuous embeddings $(H^1(\Omega))^2 \subset H_{{\rm div},\Omega,\Gamma} \subset H_{{\rm div},\Omega}$ enable us to complete the proof.
\end{proof} 
\begin{remark}
The convergence result of theorem \ref{th1d} relies on some regularity assumptions on the quasi-reversibility solution $(v_\eps,\qq_{\eps})$.
Hence, an analysis of the regularity of such solution in a polygonal domain by using the technique of \cite{grisvard} would be interesting though challenging, since the two functions
$v_\eps$ and $\qq_\eps$ are coupled by the boundary conditions. Such analysis is postponed to some future contribution.
\end{remark}
\subsection{A few numerical experiments}
We now present some numerical illustrations of the ``exterior approach" algorithm in a domain $D$ delimited by the curve defined in polar coordinates by
\be r(\theta)=1+0.1\,\sin(3\theta), \quad \theta \in [0,2\pi].\label{exterieur}\ee
Two different obstacles $O$ are considered (an easy convex case and a more difficult non convex one):
the obstacle $O^1$ defined by 
\[r(\theta)=0.5+0.1\, \cos(\theta) -0.02\,\sin(2\theta), \quad \theta \in [0,2\pi]\]
and the obstacle $O^2$ defined as the union of the disk centered at $(-0.3,-0.3)$ of radius $0.2$ and the disk centered at $(0.4,0.3)$ of radius $0.15$.
The synthetic data of our inverse problem are obtained by solving the following forward problem: for a Dirichlet data $g_D$ on $\partial D \times (0,T)$, find $u$ in $\Omega \times (0,T)$ such that
 \begin{equation}
\label{forward}
\left\{
\begin{array}{cccc}
& \partial_t u -\Delta u  = 0&  \text{in} &\Omega \times (0,T)\\
& u=g_D  & \text{on} & \partial D \times (0,T)\\
& u=0 & \text{on} &\partial O \times (0,T)\\
& u=0 & \text{on} &\Omega \times \{0\}.
\end{array}
\right.
\end{equation}
Such problem is solved numerically by using a standard 2D finite element method in space and a finite difference scheme in time. It should be noted that in order to avoid an inverse crime, the mesh used to solve the inverse problem is different from the mesh used to solve the forward problem (\ref{forward}) that provides the artificial data.
Two kinds of Dirichlet data are used:
an easy case $g^1_D(r,\theta,t) = 4t(1-t)$ and a more difficult case
\[ g^2_D(r,\theta,t) = 4t(1-t)\,\cos(\theta - 4\pi\,t).\]
The data of our inverse problems are given by $(g_0,g_1)$ on $\Gamma$, where $g_0$ is the restriction of $g_D$ on $\Gamma$ and $g_1$ is the normal derivative $\partial_\nu u$ on $\Gamma$, where $u$ is the solution to problem (\ref{forward}).
Some pointwise Gaussian noise is added to the Dirichlet data $g_0$ such that the contaminated data $g_0^\delta$ satisfies $||g_0^\delta-g_0||_{L^2(\Gamma)}=\delta$.
Concerning $\Gamma$, for the obstacle $O^1$ two situations are analyzed: the case of complete data, that is $\Gamma=\partial D$, and the case of partial data, that is $\Gamma$ is the subpart of $\partial D$ defined by (\ref{exterieur}) for $\theta \in (0,\pi/2) \cup (\pi,3\pi/2)$.
Let us now give some details about the ``exterior approach" algorithm. The initial guess $O_0$ is the circle centered at $(0,0)$ and of radius $0.8$.
Both the quasi-reversibility problem (\ref{qr_lateral_initial_it}) in step 1 and the Poisson problem (\ref{poisson_n}) in step 2 are solved on the same fixed mesh based on a polygonal domain that approximates $D$, where $D$ is given by (\ref{exterieur}).  At each step $n$, the updated domain $O_n$ is approximated by a polygonal line defined on such mesh. 
In addition, while a simple $P1$ triangular finite element is used to solve the 2D problem (\ref{poisson_n}), the tensorized finite elements described in section \ref{sec_tensor} are used to solve the 3D problem (\ref{qr_lateral_initial_it}). If not specified, the final time $T$ is equal to $1$. The size of the mesh is such that the number of segments of the polygonal line that approximates the exterior boundary $\partial D$ is around 100, while the number of time intervals is around 70 for $T=1$.
The right-hand side $f$ in problem (\ref{poisson_n}) is chosen as a sufficiently large constant which can slightly differ from one case to another and will be given in each case. 
If not specified, the parameters in problem
(\ref{qr_lateral_initial_it}) are chosen as $\eps=0.01$ and $M=20$. 
For a study of the different parameters of the ``exterior approach", in particular the selection of $f$, $\eps$ and the stopping criterion, the reader will refer to previous articles, especially \cite{bourgeois_darde1} and \cite{becache_bourgeois_darde_franceschini}.
Before testing the ``exterior approach" algorithm, let us first test the quasi-reversibility method only, that is the step 1 of the algorithm, for a known and fixed obstacle $O$.
More precisely, we are interested in the discrepancy between the solution of problem (\ref{qr_lateral_initial_it}) and the exact solution in the domain $\Omega \times (0,T)$, for obstacle $O^1$ and complete data on $\partial D$
obtained from Dirichlet data $g^1_D$. 
This discrepancy is represented in figure \ref{QR} in the space domain $\Omega$ at fixed time $t=0.5T$ for three different amplitudes of noise, that is $\delta=0$ (no noise), $\delta=0.05$ and $\delta=0.1$ and at fixed time $t=T$ for $\delta=0.1$ only (the results for $\delta=0$ and $\delta=0.05$ are almost the same when $t=T$). We observe that the quality of the solution to problem (\ref{qr_lateral_initial_it}) strongly deteriorates from the exterior boundary to the interior boundary and from $t=0.5T$ to $t=T$, which is expected since we are here concerned with the ill-posed problem of the heat equation with exterior lateral Cauchy data and initial condition (\ref{lateral_initial}). 
\begin{figure}[!h]
\centering
\includegraphics[width=0.49\textwidth]{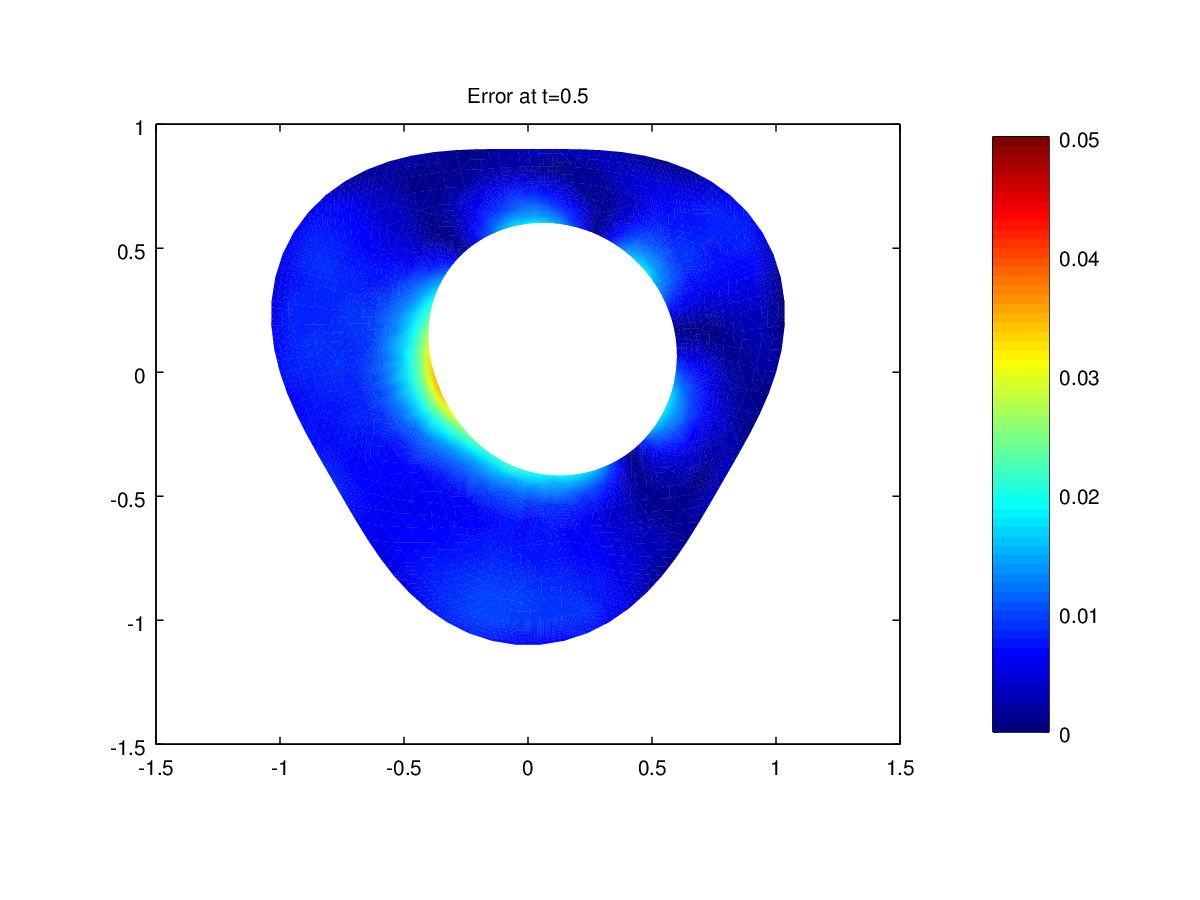}
\includegraphics[width=0.49\textwidth]{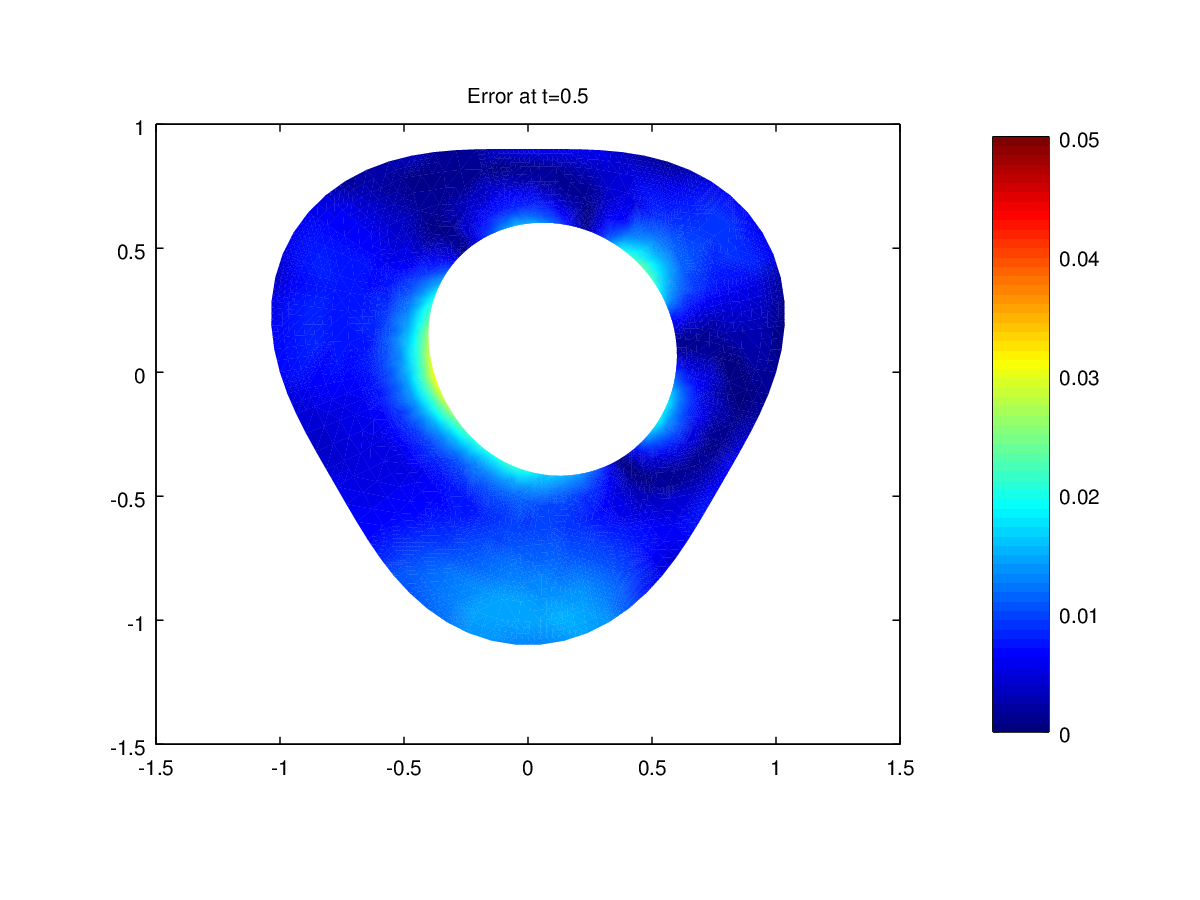} \\
\includegraphics[width=0.49\textwidth]{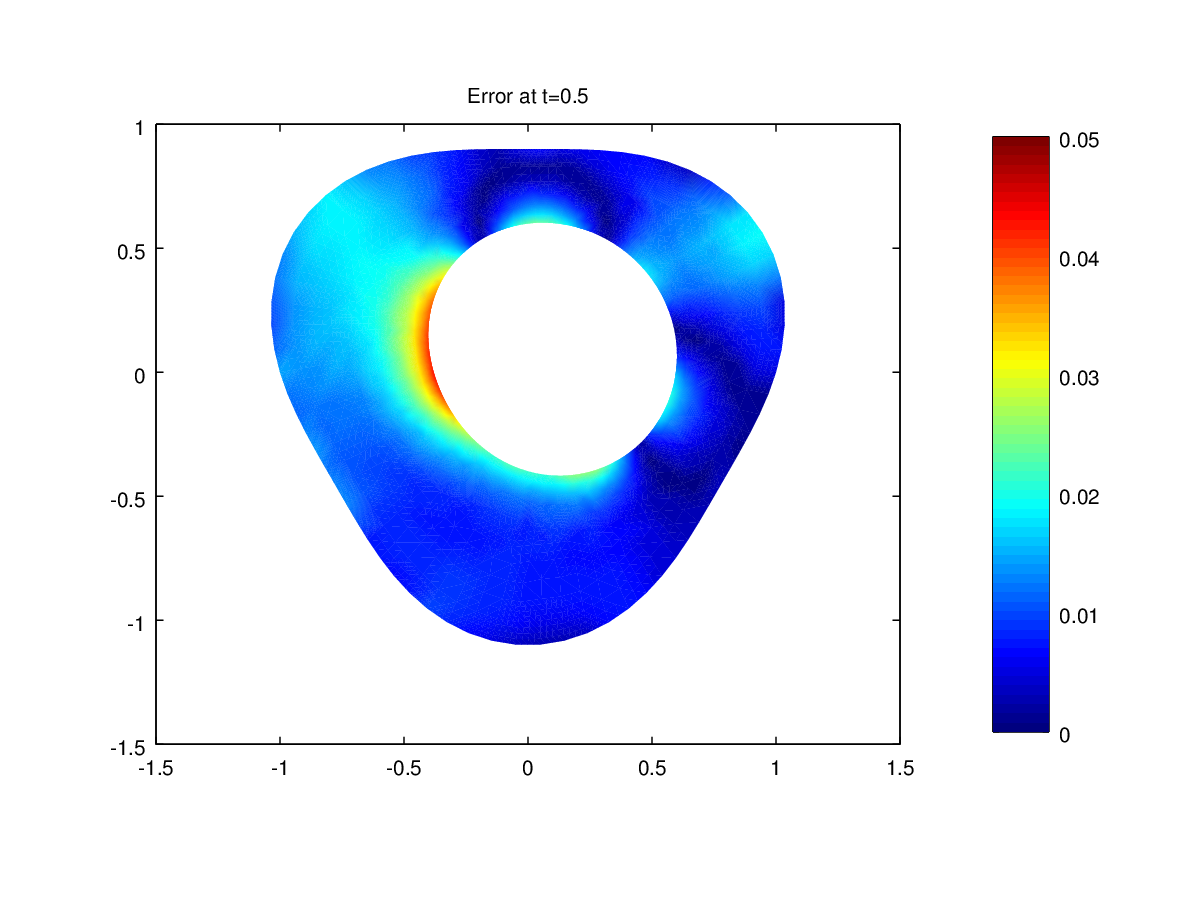}
\includegraphics[width=0.49\textwidth]{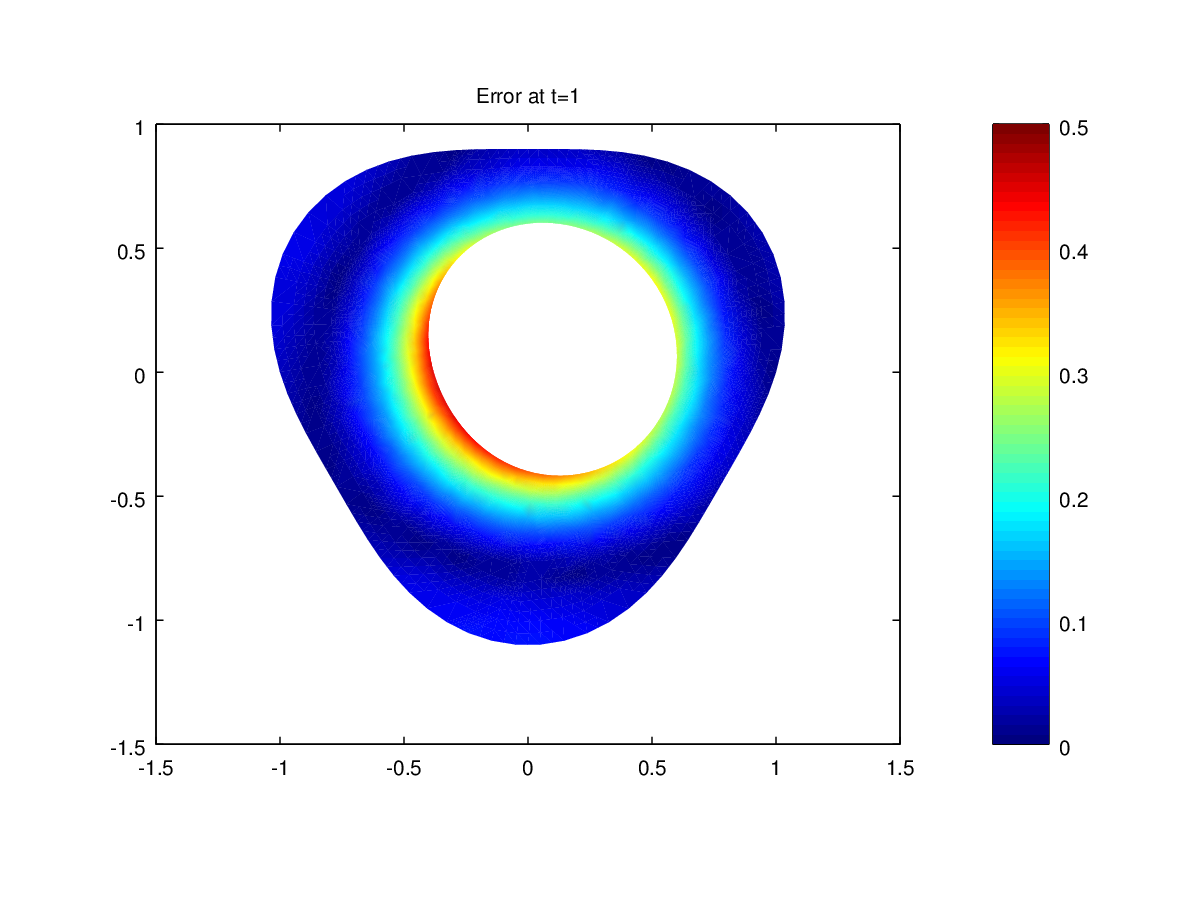}
\caption{Discrepancy between the reconstructed solution and the exact one for known obstacle $O^1$ and complete Cauchy data obtained from Dirichlet data $g^1_D$. Top left: exact data and $t=0.5T$. Top right: noisy data of amplitude $\delta=0.05$ and $t=0.5T$. Bottom left: noisy data of amplitude $\delta=0.1$ and $t=0.5T$. Bottom right: noisy data of amplitude $\delta=0.1$ and $t=T$.}
\label{QR}
\end{figure}
Now let us perform the ``exterior approach" algorithm.
Since the quality of the solution to problem (\ref{qr_lateral_initial_it}) seems unsatisfactory near $t=T$,
in step 2 of the algorithm we compute $V_n$ as $||u_n(x,\cdot)||_{L^2(0,T/2)}$ instead of $||u_n(x,\cdot)||_{L^2(0,T)}$ in order to improve the accuracy of the velocity of the level fronts. In the following numerical experiments, convergence of the sequence of obstacles $O_n$ is achieved for at least $n=10$ and at most $n=20$ iterations. 
In figure \ref{P1}, starting from the initial guess $O_0$ we have plotted the successive level fronts as well as the reconstructed obstacle $O^1$ compared to the exact one, in the case of complete Cauchy data based on the Dirichlet data $g^1_D$ (we have chosen $f=-20$).
For $T=1$, we test three different amplitudes of noise ($\delta=0$, $\delta=0.05$ and $\delta=0.1$) and for $T=0.5$ instead of $T=1$, we only consider the worst case $\delta=0.1$. We observe that the obstacle is well reconstructed, even in the presence of noisy data, which is a consequence of our relaxed formulation of quasi-reversibility which takes our noisy Cauchy data in a weak way.
\begin{figure}[!h]
\centering
\includegraphics[width=0.49\textwidth]{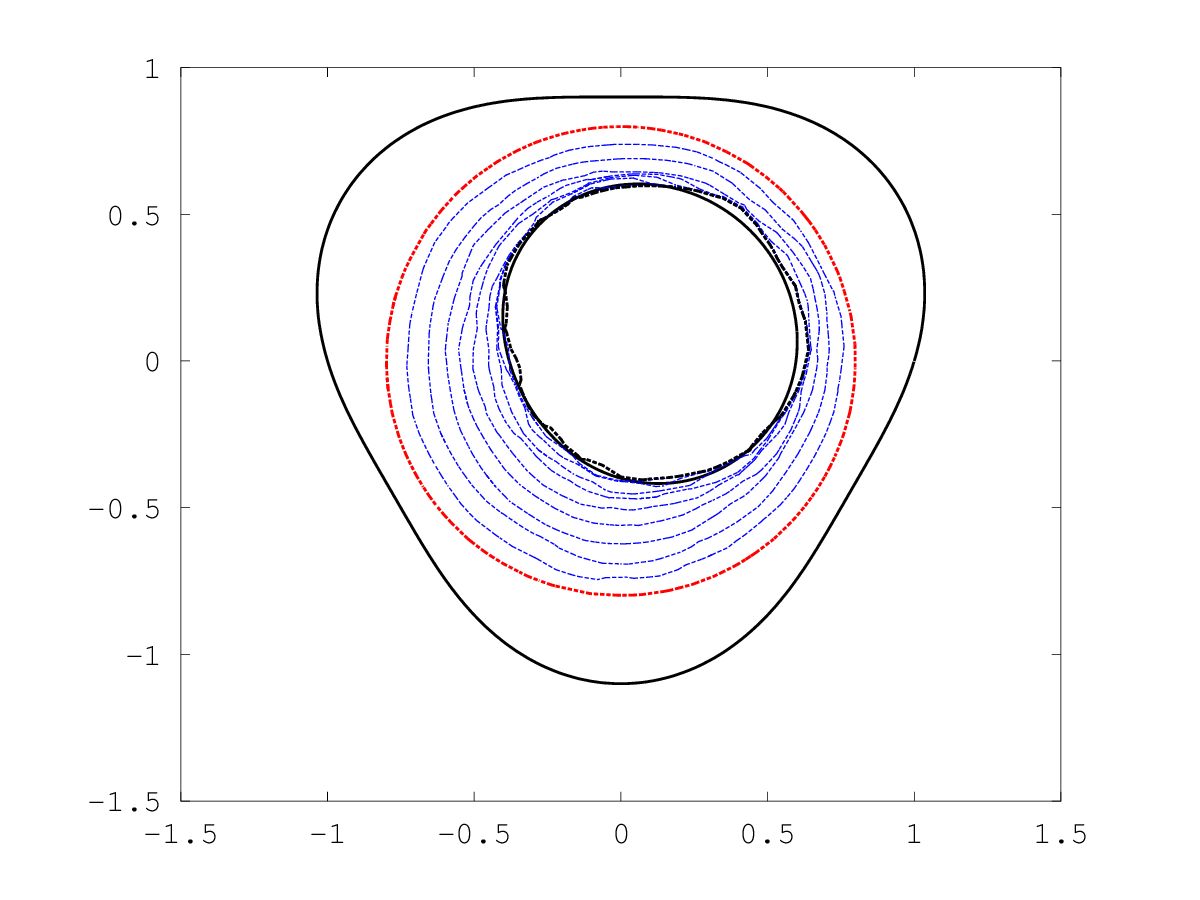}
\includegraphics[width=0.49\textwidth]{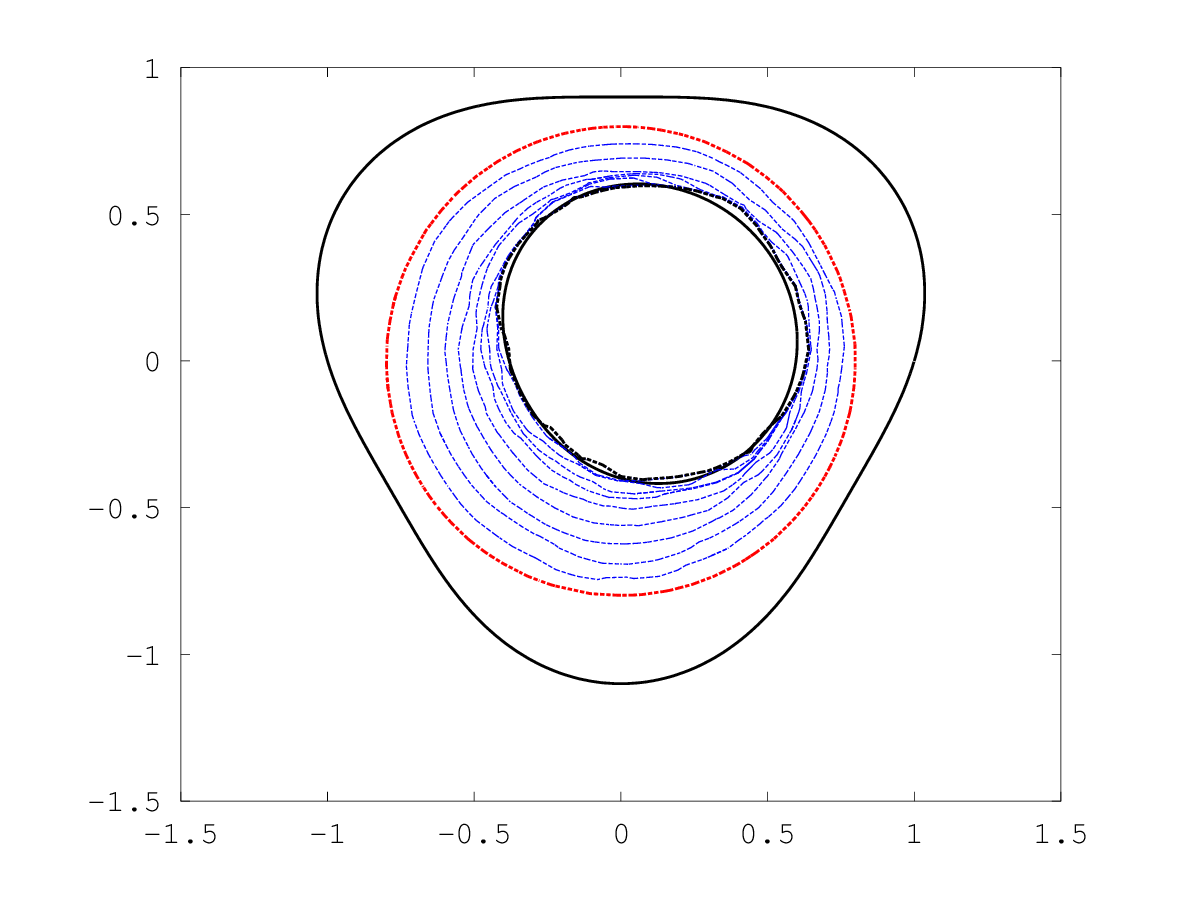} \\
\includegraphics[width=0.49\textwidth]{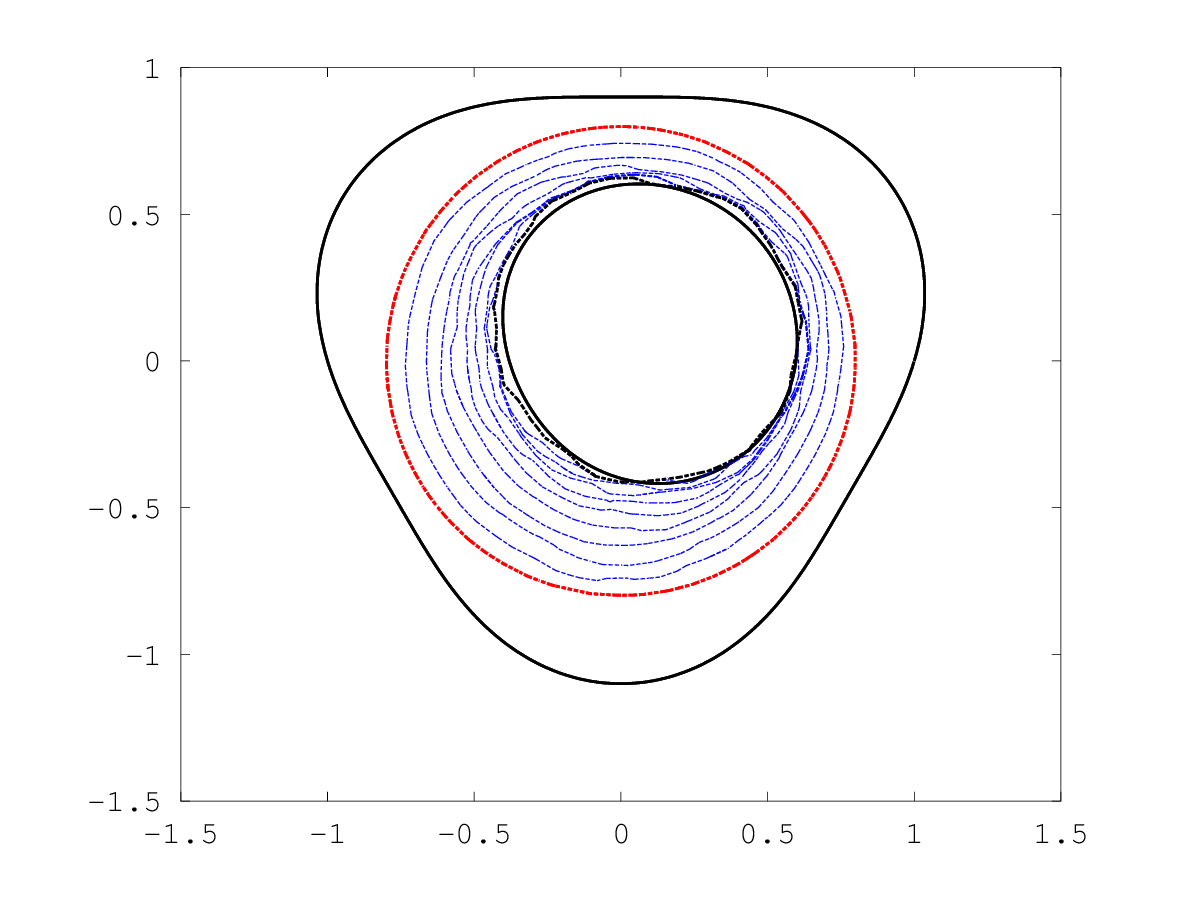}
\includegraphics[width=0.49\textwidth]{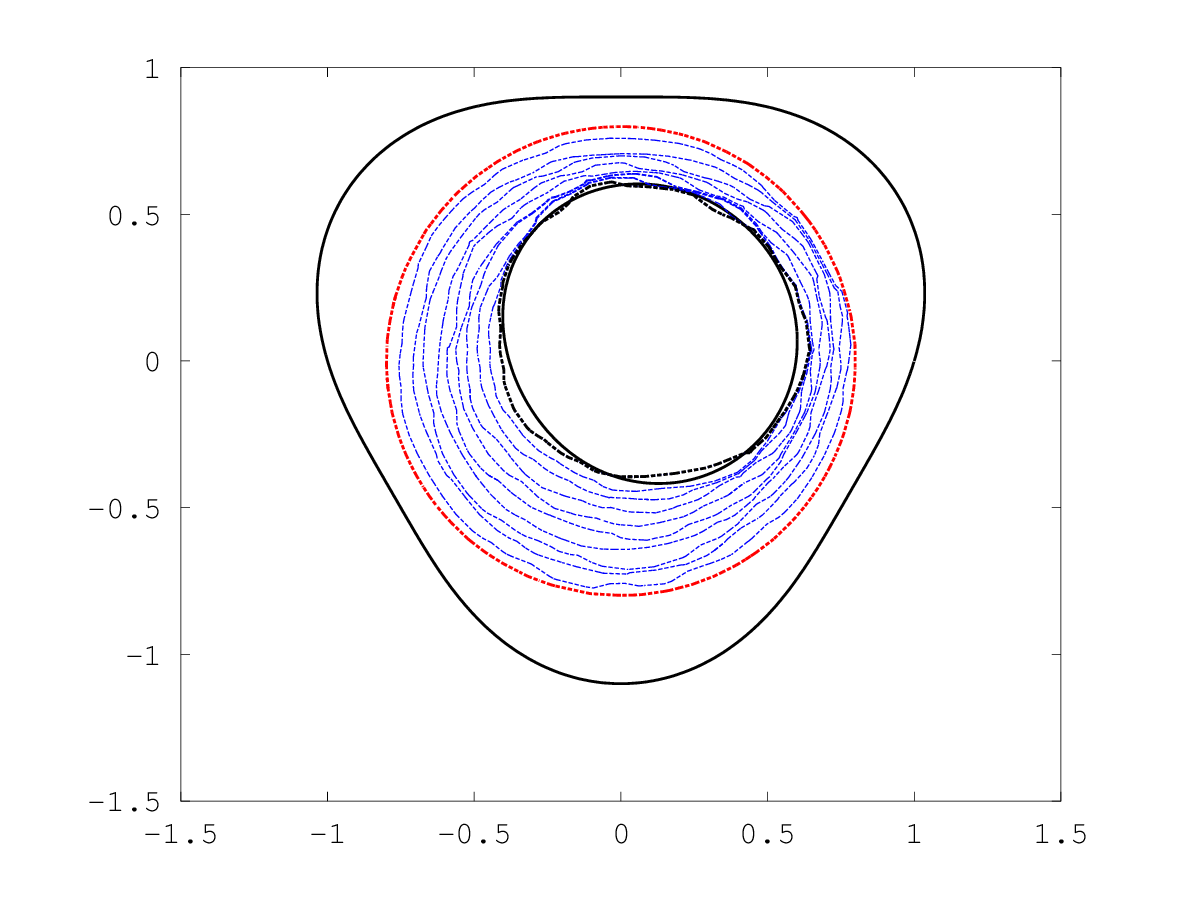}
\caption{Reconstructed obstacle $O^1$ with complete Cauchy data obtained from Dirichlet data $g^1_D$. Top left: exact data with $T=1$. Top right: noisy data of amplitude $\delta=0.05$ and $T=1$. Bottom left: noisy data of amplitude $\delta=0.1$ and $T=1$.
Bottom right: noisy data of amplitude $\delta=0.1$ and $T=0.5$.}
\label{P1}
\end{figure}
Figure \ref{P2} represents the same results as in figure \ref{P1} but in the case of complete Cauchy data based on the Dirichlet data $g^2_D$ (we have chosen $f=-15$).
The same conclusions as before can be drawn in this second case. Besides, as we observed in \cite{becache_bourgeois_darde_franceschini} for the 1D case, increasing the duration of measurements improves the quality of the identification. 
\begin{figure}[!h]
\centering
\includegraphics[width=0.49\textwidth]{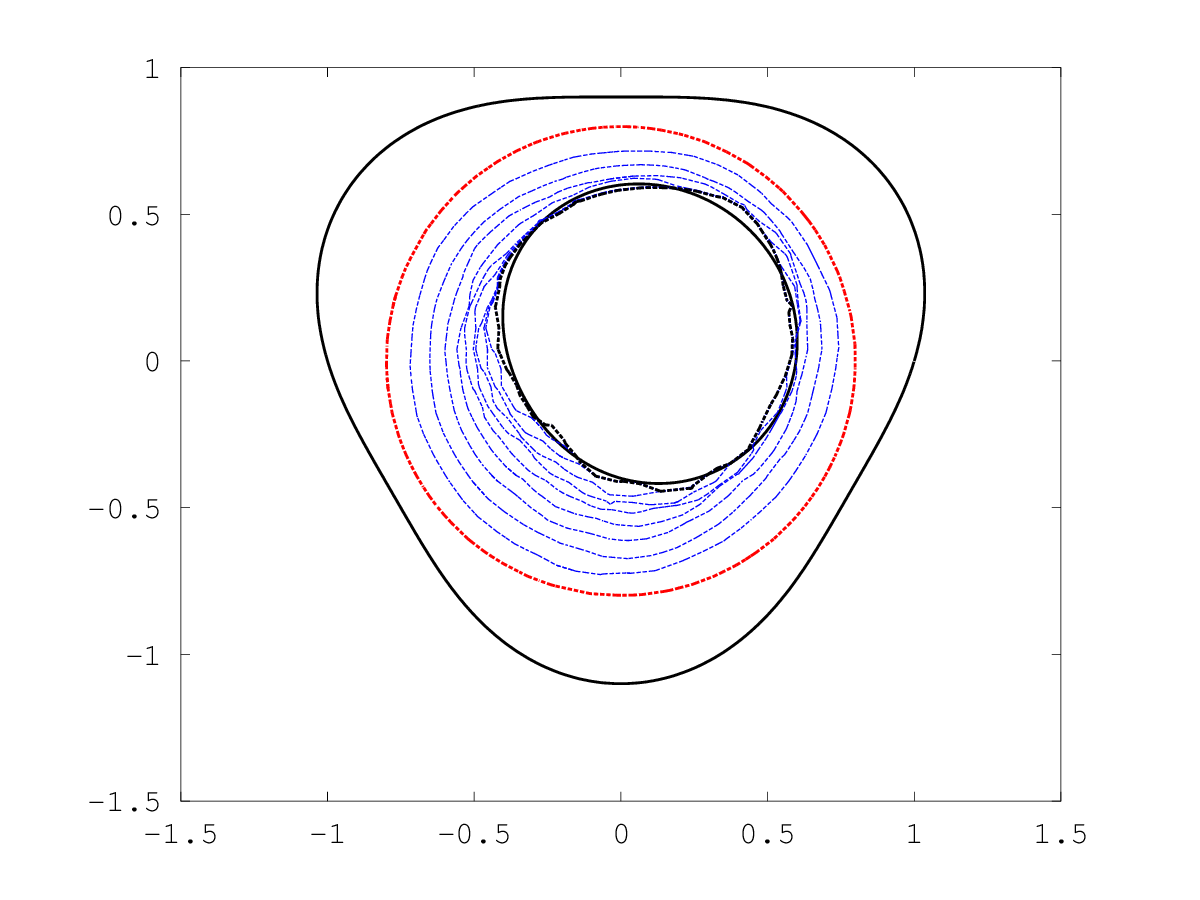}
\includegraphics[width=0.49\textwidth]{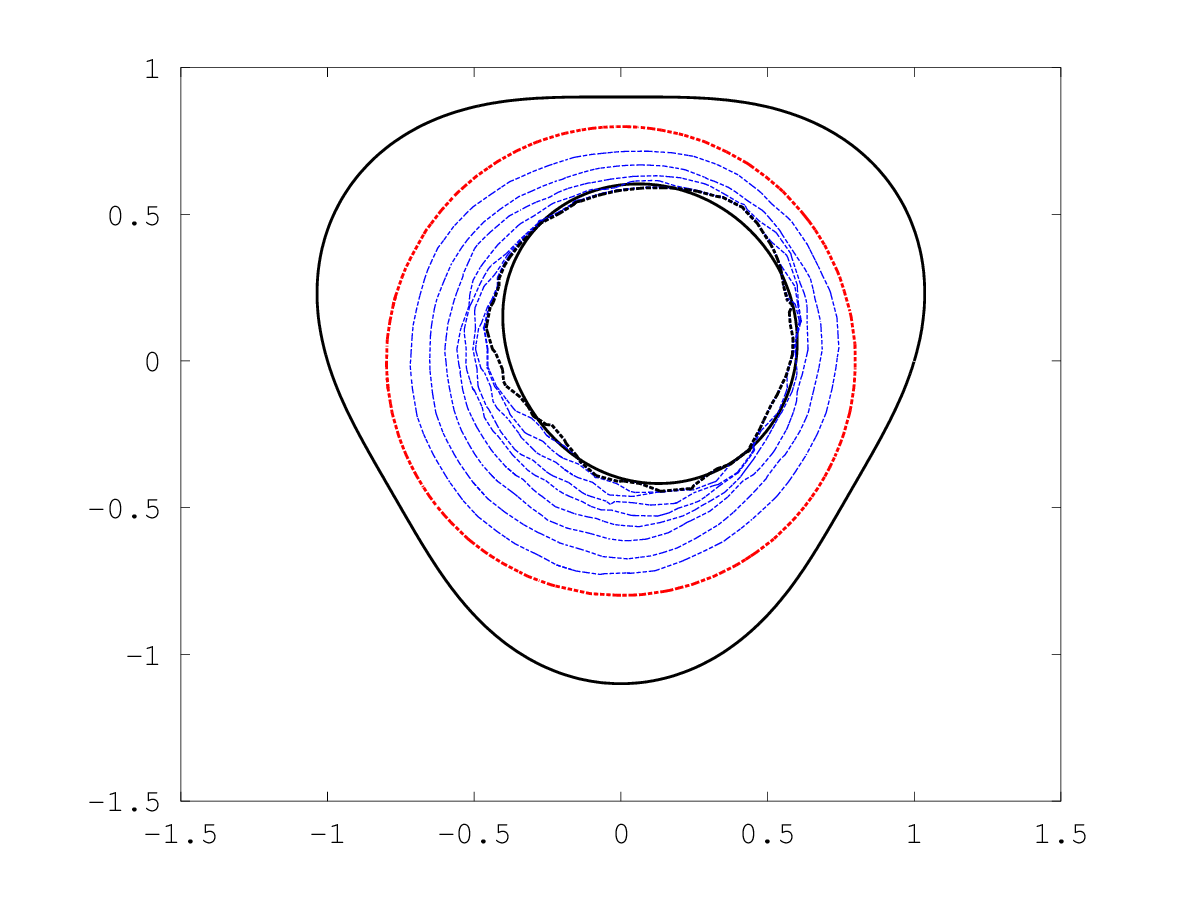} \\
\includegraphics[width=0.49\textwidth]{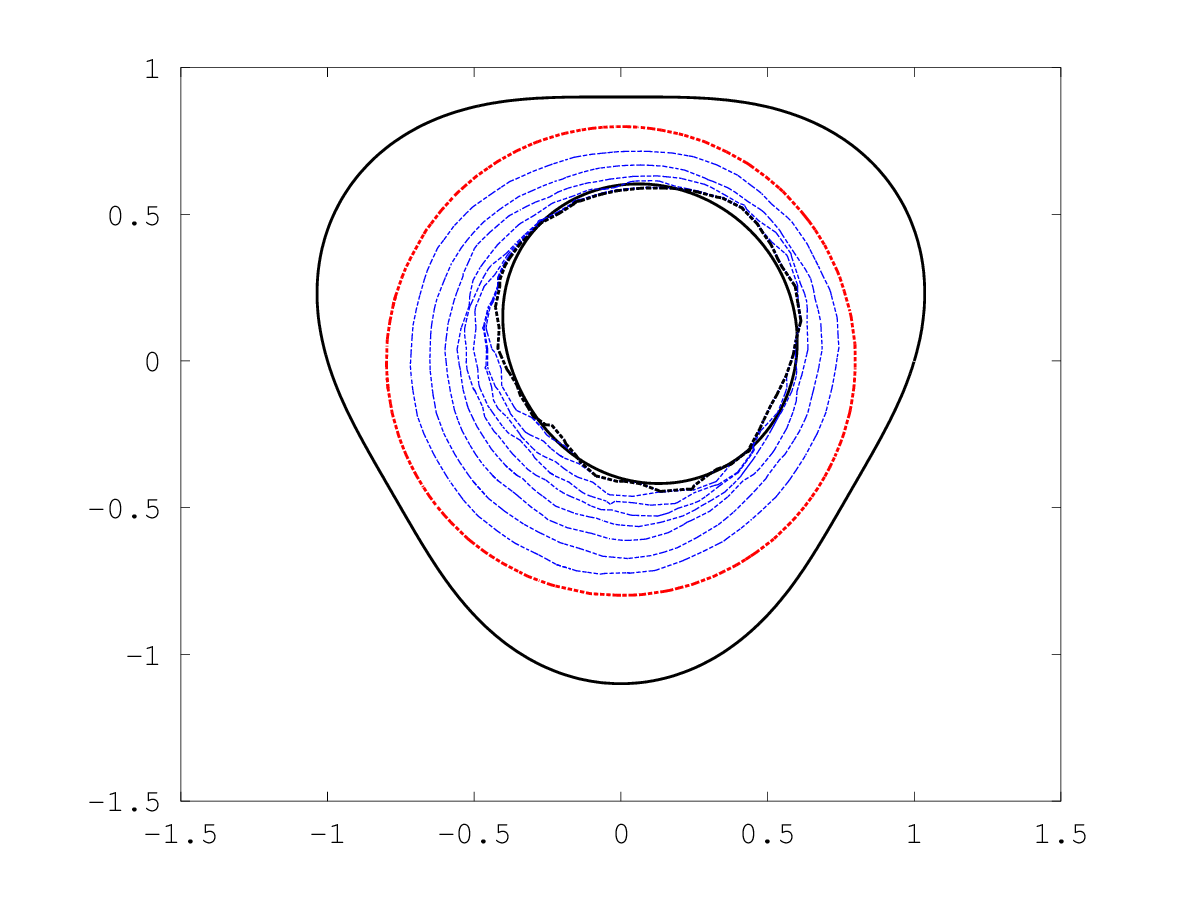}
\includegraphics[width=0.49\textwidth]{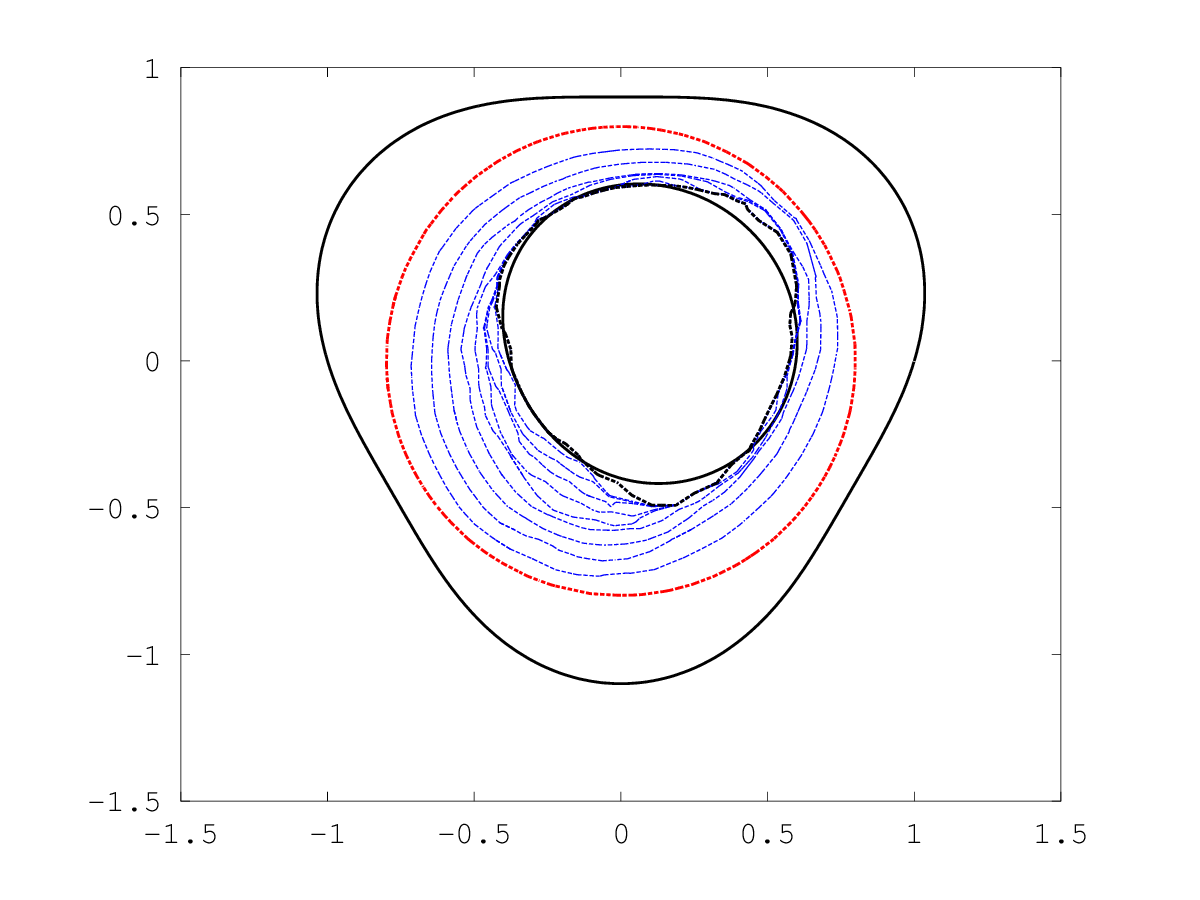}
\caption{Reconstructed obstacle $O^1$ with complete Cauchy data obtained from Dirichlet data $g^2_D$. Top left: exact data with $T=1$. Top right: noisy data of amplitude $\delta=0.05$ and $T=1$. Bottom left: noisy data of amplitude $\delta=0.1$ and $T=1$.
Bottom right: noisy data of amplitude $\delta=0.1$ and $T=0.5$.}
\label{P2}
\end{figure}
In figure \ref{P3} we reconstruct the obstacle $O^1$ with uncontaminated partial Cauchy data (instead of complete data) based on the Dirichlet data $g_D^1$ ($f=-23$) or $g_D^2$ ($f=-17$).
The obtained results have to be compared to the top left figures of \ref{P1} and $\ref{P2}$, respectively. It can be seen that the quality of the reconstructions strongly decreases, particularly for the most difficult case of data: we recall that no boundary data at all is prescribed on half of the boundary of $D$.  
\begin{figure}[!h]
\centering
\includegraphics[width=0.49\textwidth]{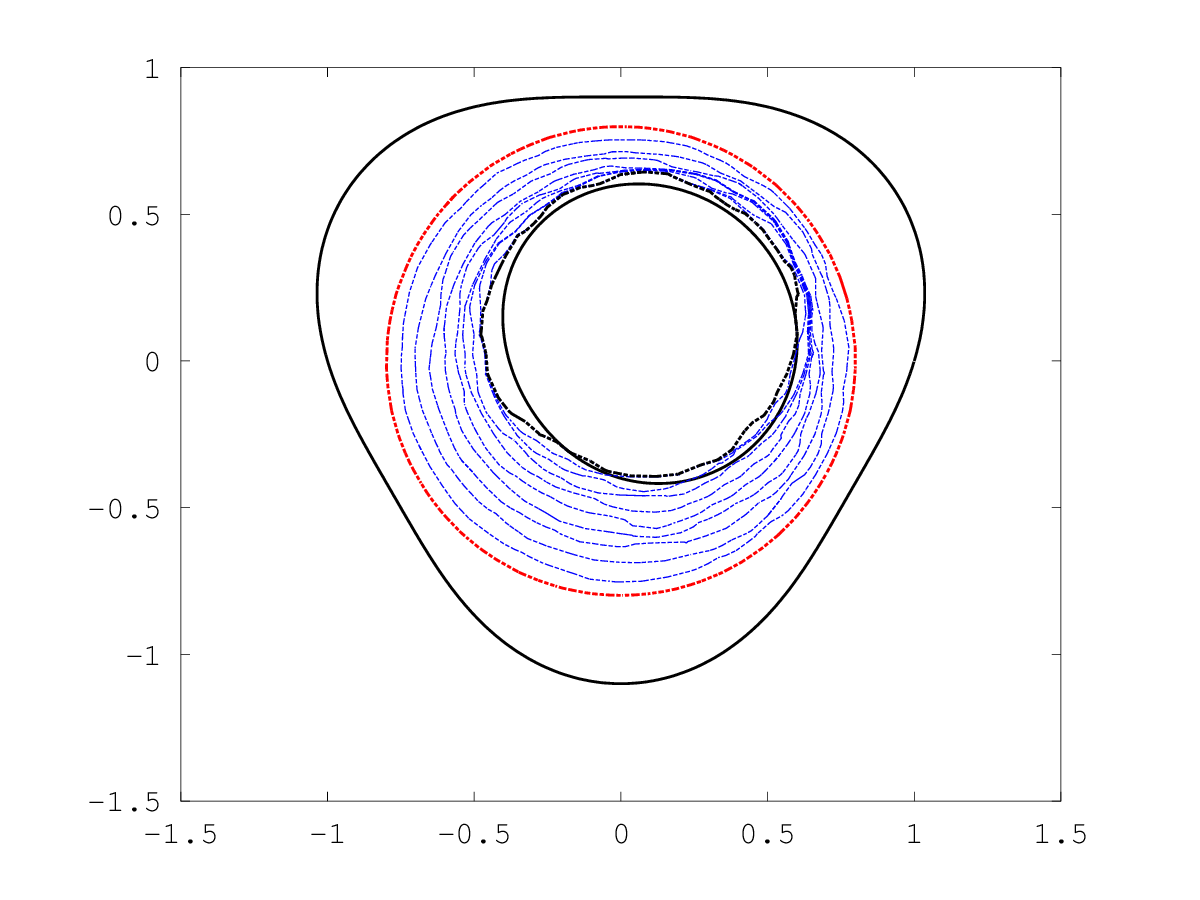}
\includegraphics[width=0.49\textwidth]{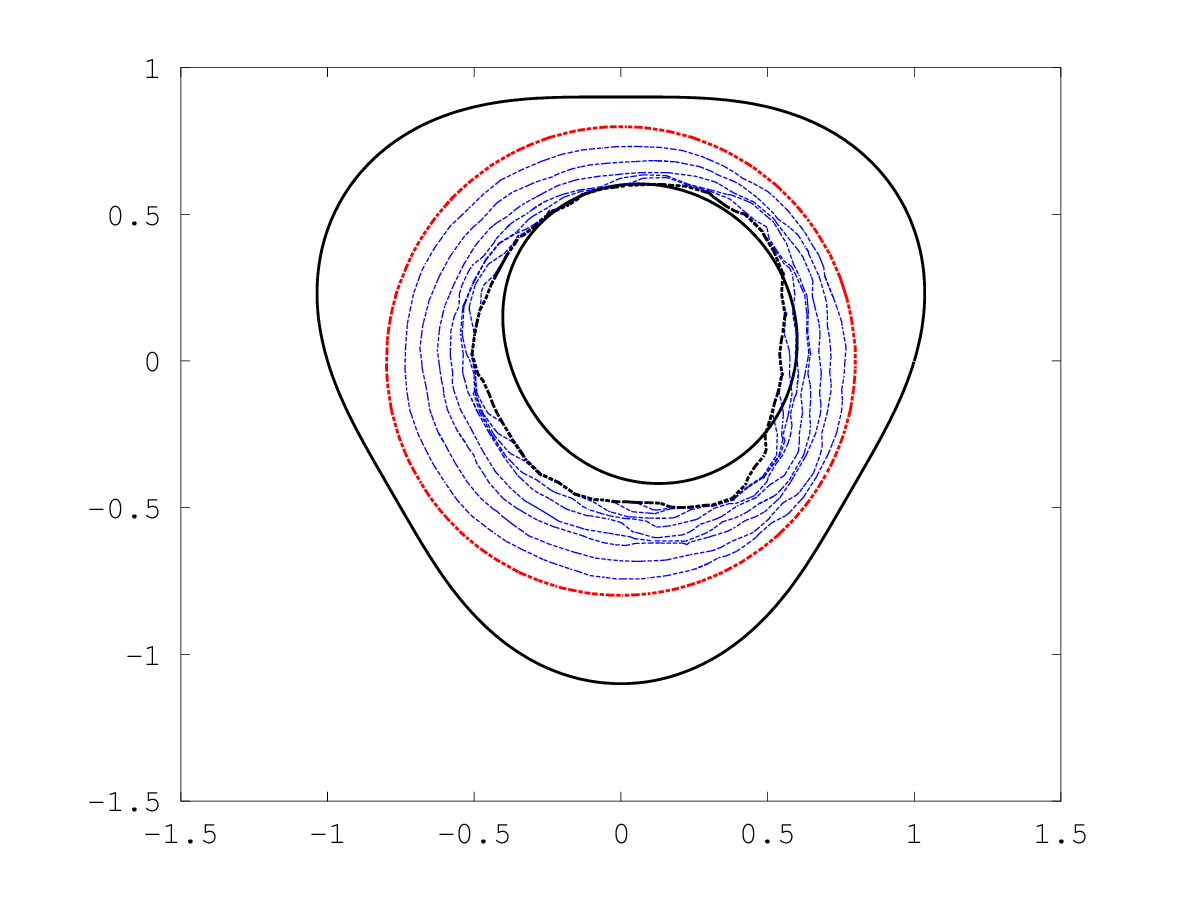}
\caption{Reconstructed obstacle $O^1$ with partial Cauchy data obtained from Dirichlet data $g_D^1$ (left) and $g_D^2$ (right).}
\label{P3}
\end{figure}
Lastly, in picture \ref{P4} we present the result of the identification of obstacle $O^2$ with complete Cauchy data based on the Dirichlet data $g_D^1$, either without noise or with noise of amplitude $\delta=0.1$ ($f=-0.14$ and $\eps=0.1$), and with the complete data based on the Dirichlet data $g_D^2$ with noise of 
amplitude $\delta=0.1$ (with the same parameters $f$ and $\eps$).
\begin{figure}[!h]
\centering
\includegraphics[width=0.49\textwidth]{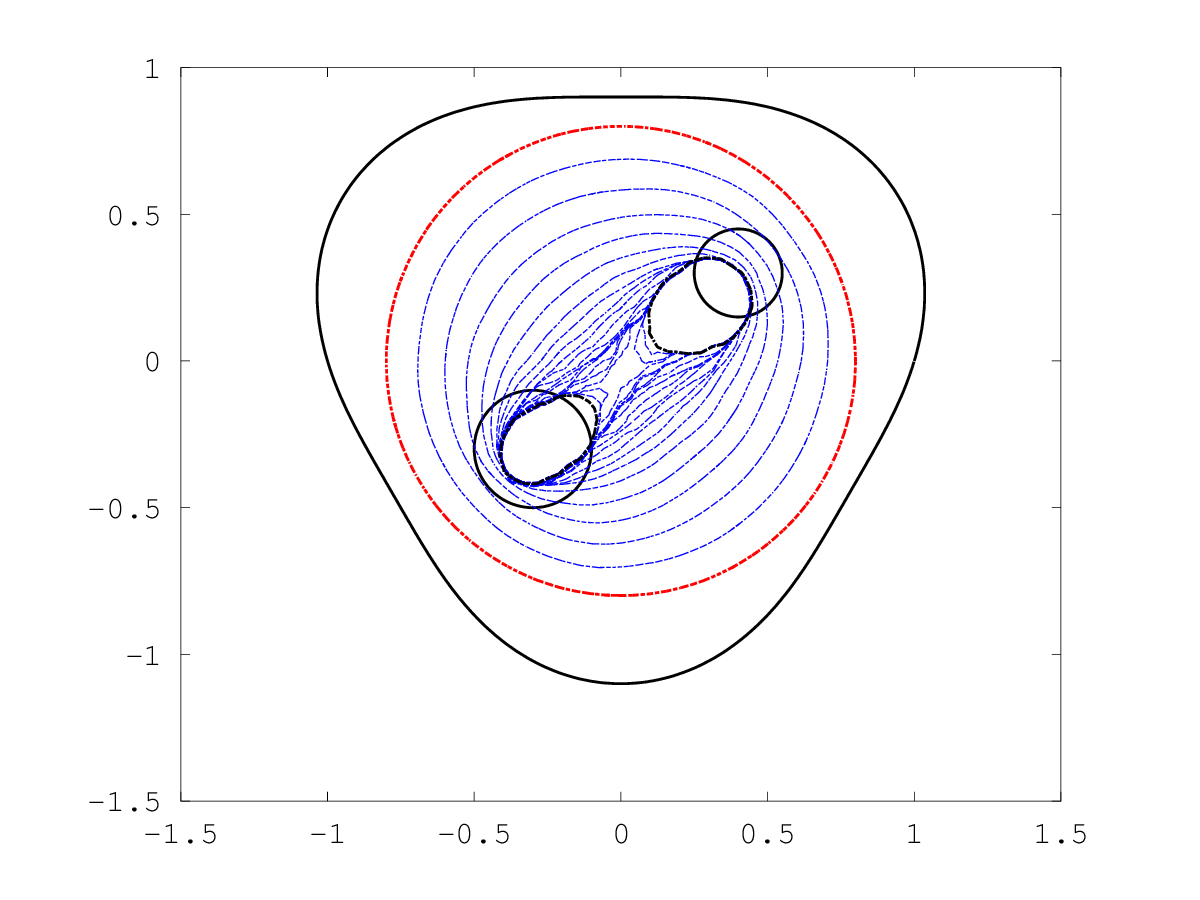}
\includegraphics[width=0.49\textwidth]{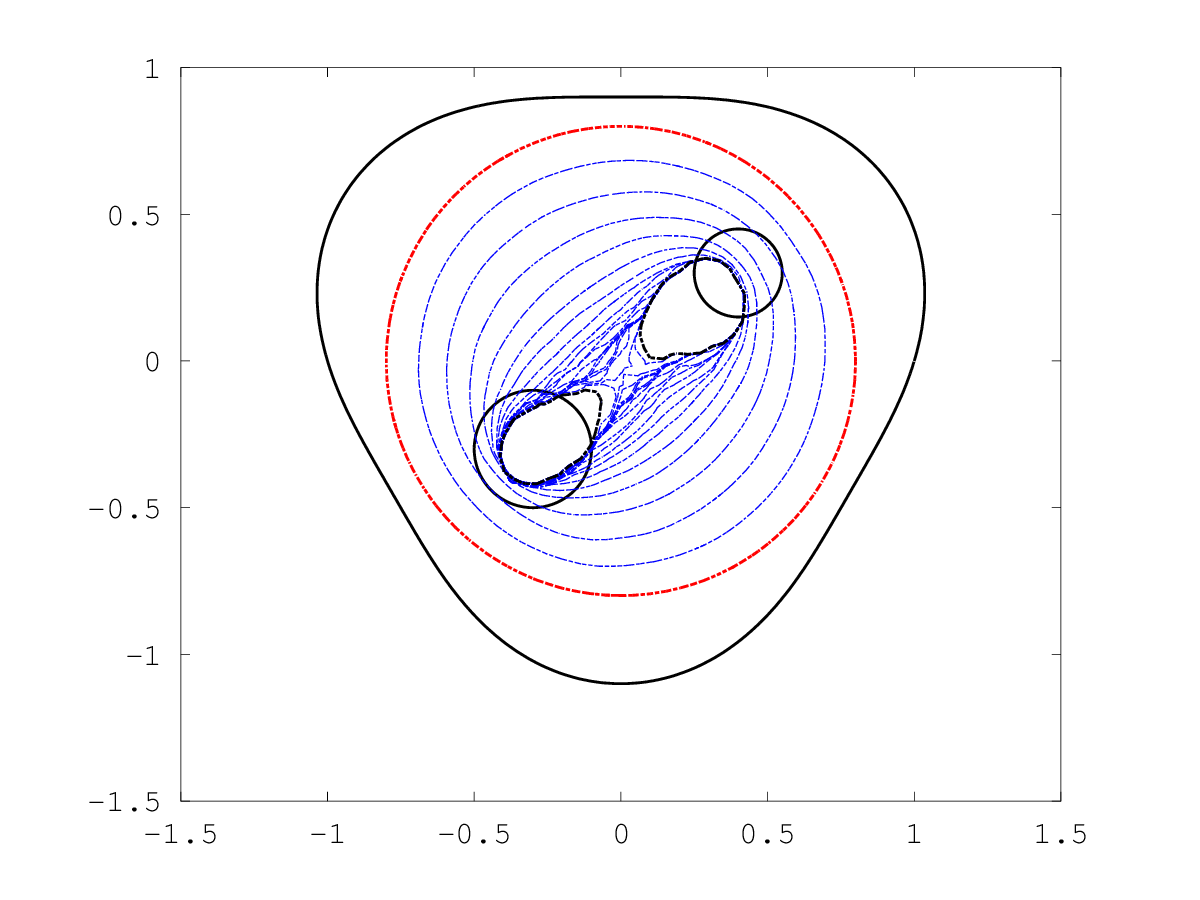}\\
\includegraphics[width=0.49\textwidth]{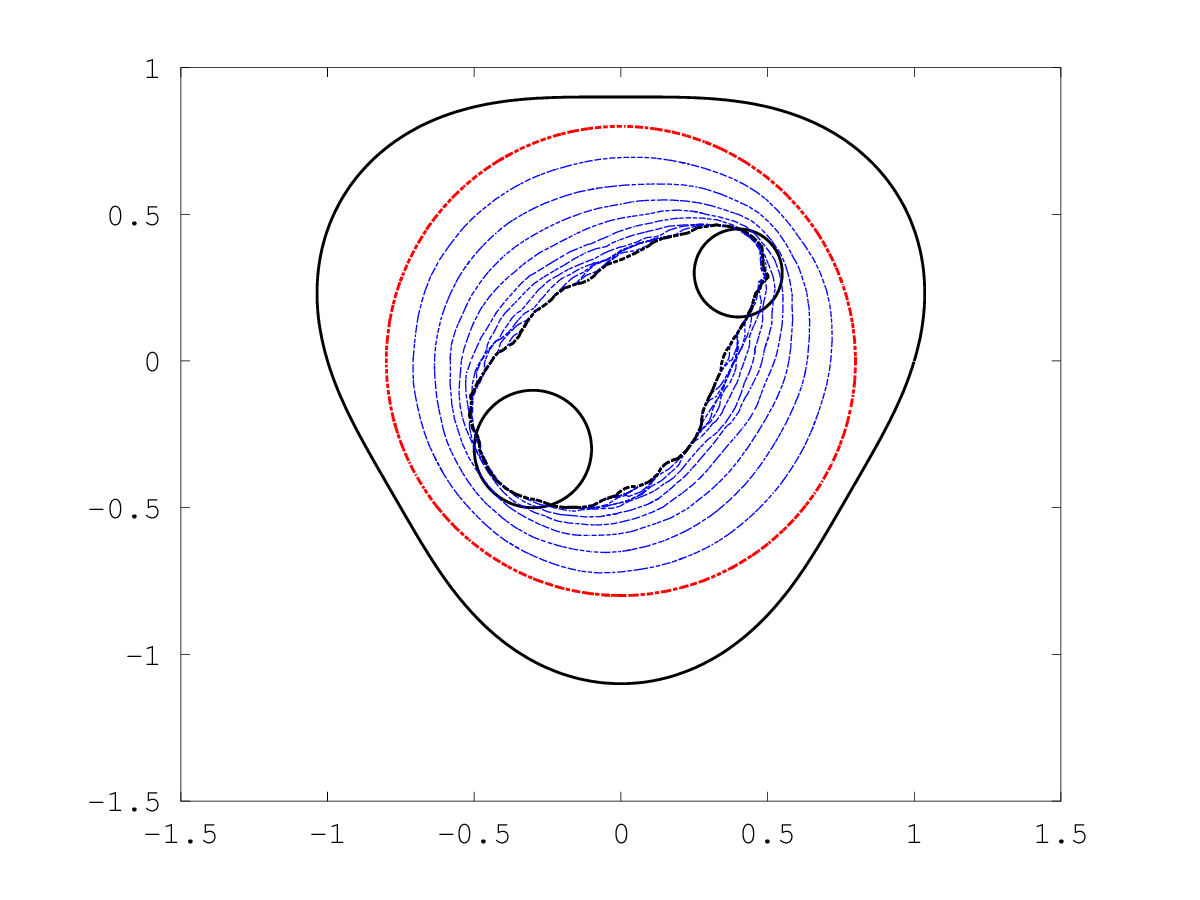}
\caption{Reconstructed obstacle $O^2$ with complete Cauchy data. Top: data are obtained from Dirichlet data $g_D^1$ without noise (left) and with noisy data of amplitude $\delta=0.1$ (right). Bottom: data are obtained from Dirichlet data $g_D^2$ with noisy data of amplitude $\delta=0.1$.}
\label{P4}
\end{figure}
\section{Conclusion and perspectives}
We have shown in this paper that our ``exterior approach" is applicable to the inverse obstacle problem for the heat equation with lateral Cauchy data and initial condition.
A specificity of our method is that those lateral Cauchy data may be known only on a subpart of the boundary while no data at all are known on the complementary part
(as in figure \ref{P3}). In addition, our ``relaxed" formulation of quasi-reversibility, which consists in taking into account our noisy boundary conditions in a weak way, 
seems quite robust with respect to the amplitude of the noise. However, if we compare our results for the heat equation and those obtained in \cite{bourgeois_darde1} for the Laplace equation, it seems that the quality of the identification is slightly worse in the first case than in the second one (see figure \ref{P4}). Maybe the ill-posedness of the inverse obstacle problem is intrinsically more severe for the heat equation than for the Laplace equation. Our aim is now to try the ``exterior approach" to solve the inverse obstacle problem for the wave equation in 2D, expecting better numerical results than for the heat equation.
\bibliography{ext_app_heat_eq}
\bibliographystyle{siam} 
\end{document}